\documentclass[12pt,reqno]{amsart}
\usepackage{etoolbox}
\makeatletter
\patchcmd\maketitle
  {\uppercasenonmath\shorttitle}
  {}
  {}{}
\patchcmd\maketitle
  {\@nx\MakeUppercase{\the\toks@}}
  {\the\toks@}
  {}
  {}{}
\patchcmd\@settitle{\uppercasenonmath\@title}{\Large}{}{}
\patchcmd\@setauthors
  {\MakeUppercase{\authors}}
  {\authors}
  {}{}
\makeatother
\usepackage{amsmath,amssymb,amsthm}
\usepackage{color}
\usepackage{url}
\usepackage{tikz-cd}
\usepackage[utf8]{inputenc}
\usepackage[T1]{fontenc}
\newtheorem{theorem}{Theorem}[section]

\newtheorem{corollary}{Corollary}[section]
\newtheorem{proposition}{Proposition}[section]
\newtheorem{lemma}{Lemma}[section]
\newtheorem{remark}{Remark}[section]
\newtheorem{example}{Example}[section]
\newtheorem{examples}{Examples}[section]
\numberwithin{equation}{section}
  \newtheorem{thqt}{Theorem}
  
  \newtheorem{lemqt}[thqt]{Lemma}

  \renewcommand{\thethqt}{\Alph{thqt}}
\usepackage[colorlinks=true]{hyperref}
\hypersetup{urlcolor=blue, citecolor=red , linkcolor= blue}
\usepackage[capitalise,noabbrev,nameinlink]{cleveref}      
\begin{document}
\author[T. Bottazzi, C. Conde and K. Feki] {\Large{Tamara Bottazzi}$^{1_{a,b}}$, \Large{Cristian Conde}$^{2_{a,b}}$ and \Large{Kais Feki}$^{3}$}

\address{$^{[1_a]}$ Universidad Nacional de R\'io Negro. LaPAC, Sede Andina (8400) S.C. de Bariloche, Argentina.}
\address{$^{[1_b]}$ Consejo Nacional de Investigaciones Cient\'ificas y T\'ecnicas, (1425) Buenos Aires,
	Argentina.}
\email{\url{tbottazzi@unrn.edu.ar}}

\address{$^{[2_a]}$ Instituto Argentino de Matem\'atica ``Alberto Calder\'on", Saavedra 15 3er. piso, (C1083ACA), Buenos Aires, Argentina}
\address{$^{[2_b]}$ Instituto de Ciencias, Universidad Nacional de Gral. Sarmiento, J. M. Gutierrez 1150, (B1613GSX) Los Polvorines, Argentina}
\email{\url{cconde@ungs.edu.ar}}

\address{$^{[3]}$ University of Sfax, Sfax, Tunisia.}
\email{\url{kais.feki@hotmail.com}}

\keywords{Positive operator, numerical radius, orthogonality, parallelism.}
\subjclass[2010]{47B65, 47A12, 46C05, 47A05..}
\date{\today}

\title[On $A$-parallelism and $A$-Birkhoff-James orthogonality of operators]
{On $A$-parallelism and $A$-Birkhoff-James orthogonality of operators}
\maketitle

\begin{abstract}
In this paper, we establish several characterizations of the $A$-parallelism of bounded linear operators with respect to the seminorm induced by a positive operator $A$ acting on a complex Hilbert space. Among other things, we investigate the relationship between $A$-seminorm-parallelism and $A$-Birkhoff-James orthogonality of $A$-bounded operators. In particular, we characterize $A$-bounded operators which satisfy the $A$-Daugavet equation. In addition,  we relate the $A$-Birkhoff-James orthogonality of operators and distance formulas and we give an explicit formula of the center mass for $A$-bounded operators. Some other related results are also discussed.
\end{abstract}

\section{Introduction and Preliminaries}\label{s1}
Let $\mathcal{B}(\mathcal{H})$ denote the $C^*$-algebra of all bounded linear operators acting on a non trivial complex Hilbert space
$\mathcal{H}$ with an inner product $\langle\cdot,\cdot\rangle$ and the corresponding norm $\|\cdot\|$. The symbol $I_{\mathcal{H}}$ stands for the identity operator on $\mathcal{H}$ (or $I$ if no confusion arises).

In all that follows, by an operator we mean a bounded linear operator. The range of every operator is denoted by $\mathcal{R}(T)$, its null space by $\mathcal{N}(T)$ and $T^*$ is the adjoint of $T$.  If $T, S\in\mathcal{B}(\mathcal{H}),$ we write  $T\geq S$ whenever $\langle Tx,x\rangle \geq \langle Sx,x\rangle$ for all $x\in \mathcal{H}$. An element $A\in \mathcal{B}(\mathcal{H})$ such that $A\geq0$ is called positive.  For every $A\geq 0$, there exists a unique positive $A^{1/2}\in \mathcal{B}(\mathcal{H})$ such that $A=(A^{1/2})^2$. For the rest of this article, we assume that $A\in\mathcal{B}(\mathcal{H})$ is a positive nonzero operator, which clearly induces the following semi-inner product
$$\langle\cdot,\cdot\rangle_{A}:\mathcal{H}\times \mathcal{H}\Rightarrow\mathbb{C},\;(x,y)\longmapsto\langle x, y\rangle_{A} :=\langle Ax, y\rangle.$$
Notice that the induced seminorm is given by $\|x\|_A=\sqrt{\langle x, x\rangle_A}$, for every $x\in \mathcal{H}$. This makes $\mathcal{H}$ into a semi-Hilbertian space. One can check that $\|\cdot\|_A$ is a norm on $\mathcal{H}$ if and only if $A$ is injective, and that $(\mathcal{H},\|\cdot\|_A)$ is complete if and only if $\mathcal{R}(A)$ is closed.  The semi-inner product $\langle\cdot,\cdot\rangle_A$ induces an inner product on the quotient space $\mathcal{H}/\mathcal{N}(A)$ defined as
$$[\overline{x},\overline{y}]=\langle Ax, y\rangle,$$
for all $\overline{x},\overline{y}\in \mathcal{H}/\mathcal{N}(A)$. Notice that $(\mathcal{H}/\mathcal{N}(A),[\cdot,\cdot])$ is not complete unless $\mathcal{R}(A)$ is a closed subset of $\mathcal{H}$. However, a canonical construction due to L. de Branges and J. Rovnyak in \cite{branrov} (see also \cite{fekilaa}) shows that the completion of $\mathcal{H}/\mathcal{N}(A)$  under the inner product $[\cdot,\cdot]$ is isometrically isomorphic to the Hilbert space $\mathcal{R}(A^{1/2})$
with the inner product
\begin{align}\label{I.1}
\langle A^{1/2}x,A^{1/2}y\rangle_{\mathbf{R}(A^{1/2})}:=\langle P_{\overline{\mathcal{R}(A)}}x, P_{\overline{\mathcal{R}(A)}}y\rangle,\;\forall\, x,y \in \mathcal{H},
\end{align}
where $P_{\overline{\mathcal{R}(A)}}$ denotes the orthogonal projection $\overline{\mathcal{R}(A)}$.

 For the sequel, the Hilbert space $\big(\mathcal{R}(A^{1/2}), \langle\cdot,\cdot\rangle_{\mathbf{R}(A^{1/2})}\big)$ will be denoted by $\mathbf{R}(A^{1/2})$.
 One can observed that by using \eqref{I.1}, it can be checked that
\begin{align*}
\langle Ax,Ay\rangle_{\mathbf{R}(A^{1/2})}= {\langle x, y\rangle}_{A}\quad\forall\, x,y \in \mathcal{H},
\end{align*}
which in turn implies that
\begin{align}\label{newhilbert}
\| Ax\|_{\mathbf{R}(A^{1/2})}=\|x\|_{A},
\end{align}
for all $x\in \mathcal{H}$. We refer the reader to \cite{acg3} and the references therein for more information concerning the Hilbert space $\mathbf{R}(A^{1/2})$.

For $T\in \mathcal{B}(\mathcal{H})$, an operator $S\in \mathcal{B}(\mathcal{H})$ is said an $A$-adjoint operator of $T$ if the identity ${\langle Tx, y\rangle}_{A} = {\langle x, Sy\rangle}_{A}$ holds for every $x, y\in \mathcal{H}$, or equivalently, $S$ is solution of the operator equation $AX= T^*A$. Notice that this kind of equation can be investigated by using the following well-known theorem due to Douglas (for its proof see \cite{Dou} or \cite{M.K.X}).
\begin{thqt}\label{doug}
If $T, S \in \mathcal{B}(\mathcal{H})$, then the following statements are equivalent:
\begin{itemize}
\item[{\rm (i)}] $\mathcal{R}(S) \subseteq \mathcal{R}(T)$.
\item[{\rm (ii)}] $TD=S$ for some $D\in \mathcal{B}(\mathcal{H})$.
\item[{\rm (iii)}] There exists  $\lambda> 0$ such that $\|S^*x\| \leq \lambda\|T^*x\|$ for all $x\in \mathcal{H}$.
\end{itemize}
If one of these conditions holds, then there exists a unique solution of the operator equation $TX=S$, denoted by $Q$, such that $\mathcal{R}(Q) \subseteq \overline{\mathcal{R}(T^{*})}$. Such $Q$ is called the reduced solution of $TX=S$.
\end{thqt}
If we denote by $\mathcal{B}_{A}(\mathcal{H})$ and $\mathcal{B}_{A^{1/2}}(\mathcal{H})$ the sets of all operators that admit $A$-adjoints and $A^{1/2}$-adjoints, respectively, then an application of Theorem \ref{doug} gives
\begin{align*}
\mathcal{B}_{A}(\mathcal{H}) = \big\{T\in\mathcal{B}(\mathcal{H})\,; \; \mathcal{R}(T^*A) \subseteq \mathcal{R}(A)\big\},
\end{align*}
and
\begin{align*}
\mathcal{B}_{A^{1/2}}(\mathcal{H}) = \big\{T\in \mathcal{B}(\mathcal{H})\,; \,\, \exists\, c>0\,;
\,\,{\|Tx\|}_{A} \le  c{\|x\|}_{A}, \,\, \forall x\in \mathcal{H}\big\}.
\end{align*}
Operators in $ \mathcal{B}_{A^{1/2}}(\mathcal{H})$ are called $A$-bounded. Notice that $\mathcal{B}_{A}(\mathcal{H})$ and $\mathcal{B}_{A^{1/2}}(\mathcal{H})$ are two subalgebras of $\mathcal{B}(\mathcal{H})$ which are, in general, neither closed nor dense in $\mathcal{B}(\mathcal{H})$ (see \cite{acg1}). Moreover, the following inclusions $\mathcal{B}_{A}(\mathcal{H})\subseteq \mathcal{B}_{A^{1/2}}(\mathcal{H})\subseteq \mathcal{B}(\mathcal{H})$ hold and are in general proper (see \cite{feki01}).

If $T\in \mathcal{B}_A(\mathcal{H})$, the reduced solution of the equation $AX=T^*A$ will be denoted by $T^{\sharp_A}$. Note that, $T^{\sharp_A}=A^\dag T^*A$. Here $A^\dag$ is the Moore-Penrose inverse of $A$. For more results concerning $T^{\sharp_A}$ see \cite{acg1,acg2}.

Further, $\langle\cdot,\cdot\rangle_{A}$ induces the following seminorm on $\mathcal{B}_{A^{1/2}}(\mathcal{H})$
\begin{equation}\label{semii}
\|T\|_A:=\sup_{\substack{x\in \overline{\mathcal{R}(A)},\\ x\not=0}}\frac{\|Tx\|_A}{\|x\|_A}=\sup\big\{{\|Tx\|}_A\,; \,\,x\in \mathcal{H},\, {\|x\|}_A =1\big\}<\infty.
\end{equation}
It can be observed that for $T\in\mathcal{B}_{A^{1/2}}(\mathcal{H})$, $\|T\|_A=0$ if and only if $AT=0$. Notice that it was proved in \cite{fg} that for $T\in\mathcal{B}_{A^{1/2}}(\mathcal{H})$ we have
\begin{equation}\label{newsemi}
\|T\|_A=\sup\left\{|\langle Tx, y\rangle_A|\,;\;x,y\in \mathcal{H},\,\|x\|_{A}=\|y\|_{A}= 1\right\}.
\end{equation}
It can be verified that, for $T\in\mathcal{B}_{A^{1/2}}(\mathcal{H})$,
we have ${\|Tx\|}_{A}\leq {\|T\|}_{A}{\|x\|}_{A}$ for all $x\in \mathcal{H}$.
This implies that, for $T, S\in\mathcal{B}_{A^{1/2}}(\mathcal{H})$, we have ${\|TS\|}_{A}\leq {\|T\|}_{A}{\|S\|}_{A}$.
Notice that it may happen that ${\|T\|}_A = + \infty$
for some $T\in\mathcal{B}(\mathcal{H})$ (see \cite{feki01}). For more details concerning $A$-bounded operators, see \cite{acg3} and the references therein.

Recently, A. Saddi generalized in \cite{saddi} the concept of the numerical radius of Hilbert space operators and defined the $A$-numerical radius of an operator $T\in \mathcal{B}(\mathcal{H})$ by
\begin{equation}\label{Aradius}
\omega_A(T)=\sup\{|\langle Tx, x\rangle_A|: x\in \mathcal{H}, \|x\|_A=1\}.
\end{equation}
If $T\in\mathcal{B}_{A^{1/2}}(\mathcal{H})$ then $\omega_A(T)<+\infty$ and
\begin{equation}\label{equiomega}
\frac12 \|T\|_A\leq \omega_A(T)\leq \|T\|_A.
\end{equation}

Recently, The $A$-Davis-Wielandt radius of and operator $T\in\mathcal{B}(\mathcal{H})$ is defined, as in \cite{fekisidha2019}, by
\begin{align*}
d\omega_A(T)
&=\displaystyle\sup\left\{\sqrt{|\langle Tx, x\rangle_A|^2+\|Tx\|_A^4}\,;\;x\in \mathcal{H},\;\|x\|_A=1\right\}.
\end{align*}

Notice that it was shown in \cite{fekisidha2019}, that for $T\in\mathcal{B}(\mathcal{H})$, $d\omega_A(T)$ can be equal to $+\infty$. However, if $T\in\mathcal{B}_{A^{1/2}}(\mathcal{H})$, then we have
$$\max\{\omega_A(T),\|T\|_A^2\}\leq d\omega_A(T) \leq \sqrt{\omega_A(T)^2+\|T\|_A^4}<\infty.$$

Recall that an operator $T\in\mathcal{B}(\mathcal{H})$ is said to be $A$-selfadjoint if $AT$ is selfadjoint, that is, $AT = T^*A$. Observe that if $T$ is $A$-selfadjoint, then $T\in\mathcal{B}_A(\mathcal{H})$. However, it does not hold, in general, that $T = T^{\sharp_A}$. More precisely, if $T\in\mathcal{B}_A(\mathcal{H})$, then $T = T^{\sharp_A}$ if and only if $T$ is $A$-selfadjoint and $\mathcal{R}(T) \subseteq \overline{\mathcal{R}(A)}$ (see \cite[Section 2]{acg1}). Further, an operator $T\in\mathcal{B}_A(\mathcal{H})$ is called $A$-normal if $TT^{\sharp_A} = T^{\sharp_A}T$ (see \cite{bakfeki04}).
It is obvious that every selfadjoint operator is normal. However, an $A$-selfadjoint operator is not necessarily $A$-normal (see \cite[Example 5.1]{bakfeki04}).

Now, let $\mathbb{T}$ denote the unit cycle of the complex plane, i.e. $\mathbb{T}=\{\lambda\in \mathbb{C}\,;\;|\lambda|=1 \}$.

 Recall from \cite{fekisidha2019} that an operator $T\in \mathcal{B}_{A^{1/2}}(\mathcal{H})$ is said to be $A$-norm-parallel to an operator $S\in \mathcal{B}_{A^{1/2}}(\mathcal{H})$, in short $T\parallel_A S$, if there exists such that $\|T+\lambda S\|_A=\|T\|_A+\|S\|_A$.

For $T\in\mathcal{B}\left( \mathcal{H}\right)$, the $A$-numerical range of $T$ is defined, as in \cite{bakfeki01}, by
  \begin{align*}
  W_A(T) = \left\{ {\left\langle {Tx, x} \right\rangle_A\,;\;x
    \in \mathcal{H},\left\| x \right\|_A = 1} \right\}.
\end{align*}

Recently, the concept of the $A$-spectral radius of $A$-bounded operators has been introduced in \cite{feki01} as follows:
\begin{equation}\label{newrad}
r_A(T):=\displaystyle\inf_{n\geq 1}\|T^n\|_A^{\frac{1}{n}}=\displaystyle\lim_{n\to\infty}\|T^n\|_A^{\frac{1}{n}}.
\end{equation}
We note here that the second equality in \eqref{newrad} is also proved in \cite[Theorem 1]{feki01}. An operator $T\in \mathcal{B}_{A^{1/2}}(\mathcal{H})$ is said to be $A$-normaloid if $r_A(T)=\|T\|_A$. Moreover, $T$ is called $A$-spectraloid if $r_A(T)=\omega_A(T)$. It was shown in \cite{feki01} that for every $A$-normaloid operator $T\in \mathcal{B}_{A^{1/2}}(\mathcal{H})$ we have
\begin{equation}\label{Anormaloid}
r_A(T)=\omega_A(T)=\|T\|_A.
\end{equation}
So every $A$-normaloid operator is $A$-spectraloid. The following lemma will be used in due course of time. Notice that the proof of the assertion $(i)$ can be found in \cite{acg3}. Further, for the proof of the assertions $(ii)$ and $(iii)$ we refer to \cite{feki01}. Also, the assertion $(iv)$ has been proved in \cite{majsecesuci}. Finally, the proof of last assertion can be found in \cite{fekisidha2019}.
\begin{lemma}\label{lem2}
	Let $T\in \mathcal{B}(\mathcal{H})$. Then $T\in \mathcal{B}_{A^{1/2}}(\mathcal{H})$ if and only if there exists a unique $\widetilde{T}\in \mathcal{B}(\mathbf{R}(A^{1/2}))$ such that $Z_AT =\widetilde{T}Z_A$. Here, $Z_{A}: \mathcal{H} \rightarrow \mathbf{R}(A^{1/2})$ is
	defined by $Z_{A}x = Ax$. Moreover, the following properties hold
\begin{itemize}
  \item [(i)] $\|T\|_A=\|\widetilde{T}\|_{\mathcal{B}(\mathbf{R}(A^{1/2}))}$.
  \item [(ii)] $r_A(T)=r(\widetilde{T})$.
  \item [(iii)] $\overline{W_A(T)}=\overline{W(\widetilde{T})}$.
  \item [(iv)]  $\widetilde{T^{\sharp_A}}=(\widetilde{T})^*.$
 \item [(v)] If $T,S\in \mathcal{B}_{A^{1/2}}(\mathcal{H})$, then $T\parallel_AS$ if and only if $\widetilde{T}\parallel \widetilde{S}$.
\end{itemize}
\end{lemma}
Recently, several results covering some classes of Hilbert space operators where extended to $A$-bounded operators, see \cite{fekilaa,feki01,feki03,fekisidha2019,majsecesuci,zamani3,zamani2019_2} and the references therein.

The remainder of the paper is organized as follows. Section \ref{s2} we present different characterization of  notion of $A$-seminorm-parallelism and in particular we investigate when the $A$-Davis-Wielandt radius of and operator coincides with its upper bound. In section \ref{s3},  we  give another characterizations of $A$-seminorm-parallelism related to $A$-Birkhoff-James orthogonality. Finally, section \ref{s4} is devoted to obtain some formulas for the $A$-center of mass of $A$-bounded operators  using well-known distance formulas.

\section{$A$-seminorm-parallelism}\label{s2}
We start our work with the following examples of seminorm-parallelism in semi-Hilbert spaces.
\begin{examples}
\begin{itemize}
  \item [(1)] Let $T,S\in \mathcal{B}_{A^{1/2}}(\mathcal{H})$ be linearly dependent operators. Then $T \parallel_A S$ (see \cite[Example 3]{fekisidha2019}).
  \item [(2)] Let $A=\begin{pmatrix}1&0\\0&2\end{pmatrix}$ and $T=\begin{pmatrix}
1 & 0\\
0 & -1
\end{pmatrix}$ be operators acting on $\mathbb{C}^2$. Then for $\lambda=1$, simple computations show that
\begin{align*}
\|T + \lambda I\|_A = \|T\|_A+\|I\|_A = 2.
\end{align*}
Hence $T \parallel_A I.$
\item [(3)]Let $\lambda>0$ and $A, T, S:\ell^2(\mathbb{N})\to \ell^2(\mathbb{N})$ be such that
$$S(\overline{x})=(\lambda x_1, \lambda x_2, x_3,  x_4, \ldots), \quad T(\overline{x})=(0, \lambda x_2,  x_3,  x_4, \ldots)$$
and
$$A(\overline{x})=(0, x_2, 0, 0, \ldots),$$
for every $\overline{x}=(x_1, x_2, \ldots, x_n, \ldots)\in \ell^2(\mathbb{N})$,  where $\mathbb{N}$ denotes the set of all positive integers. Clearly, $A\geq 0.$ Further, it can be observed that $\|T\|_A=\|S\|_A=\lambda.$ Now, let $\{e_j\}_{j\in \mathbb{N}}$ be the canonical orthogonal basis of $\mathcal{H}=\ell^2(\mathbb{N}).$ Then, we have
$$\|(T+S)(e_2)\|_A^2=4\lambda^2.$$
Thus, $2\lambda\leq\|T+S\|_A\leq \|T\|_A+\|S\|_A=2\lambda.$ Therefore $T\parallel_A S.$
\end{itemize}
\end{examples}

In the following proposition we state some basic properties of operator seminorm-parallelism in $\mathcal{B}_A(\mathcal{H})$.

\begin{proposition}\label{pr.11}
Let $T, S\in \mathcal{B}_{A^{1/2}}(\mathcal{H})$. The following statements are equivalent:
\begin{itemize}
  \item [(1)] $T\parallel_A S$.
  \item [(2)] $\alpha T\parallel_A \alpha S$ for every $\alpha\in\mathbb{C}\setminus\{0\}$.
  \item [(3)] $\beta T\parallel_A \gamma S$ for every $\beta, \gamma\in\mathbb{R}\setminus\{0\}$
\end{itemize}
\end{proposition}
\begin{proof}
Notice that equivalence (1)$\Leftrightarrow$(2) follows immediately from the definition of $A$-operator parallelism.\\
(1)$\Rightarrow$(3) Assume that $T\parallel_A S$. Thus $\|T+\lambda S\|_A=\|T\|_A+\|S\|_A$ for some $\lambda\in\mathbb{T}$. Let $\beta, \gamma\in\mathbb{R}\setminus\{0\}$. We suppose that $\beta \geq \gamma>0$. Hence, we see that
\begin{align*}
\|\beta T\|_A+\|\gamma S\|_A
&\geq\|\beta T+\lambda(\gamma S)\|_A\\
&=\|\beta(T+\lambda S)-(\beta-\gamma)(\lambda S)\|_A\\
&\geq \|\beta(T+\lambda S)\|_A-\|(\beta-\gamma)\lambda S\|_A\\
&=\beta\|T+\lambda S\|_A-(\beta-\gamma)\|S\|_A\\
&=\beta(\|T\|_A+\|S\|_A)-(\beta-\gamma)\|S\|_A\\
&=\|\beta T\|_A+\|\gamma S\|_A.
\end{align*}
So, $\|\beta T+\lambda(\gamma S)\|_A=\|\beta T\|_A+\|\gamma S\|_A$ for some $\lambda\in\mathbb{T}$. Therefore $\beta T\parallel_A \gamma S$.\\
(3)$\Rightarrow$(1) is trivial.
\end{proof}

The following lemma is useful in the sequel.
\begin{lemma}\label{lem3}
Let $T, S\in \mathcal{B}_{A^{1/2}}(\mathcal{H})$. Then the following statements are equivalent:
\begin{itemize}
  \item [(i)] $T\parallel_A S$.
  \item [(ii)] There exist a sequence of $A$-unit vectors $\{x_n\}$ in $\mathcal{H}$ and $\lambda\in\mathbb{T}$ such that
\begin{equation*}
\lim_{n\rightarrow\infty} \langle Tx_n, Sx_n\rangle_A=\lambda\|T\|_A\,\|S\|_A.
\end{equation*}
\end{itemize}
\end{lemma}
In order to prove Lemma \ref{lem3} we need the following result.
\begin{thqt}(\cite{fekisidha2019})\label{main1}
Let $T,S\in \mathcal{B}_{A^{1/2}}(\mathcal{H})$. Then, $T\parallel_AS$ if and only if there exists a sequence of $A$-unit vectors $\{x_n\}$ in $\mathcal{H}$ such that
    \begin{equation}\label{ThmB}
    \lim_{n\to \infty}|\langle T x_n, Sx_n\rangle_A|=\|T\|_A\|S\|_A.
    \end{equation}
\end{thqt}

\begin{remark}\label{remparallelism}
In addition, if $\|T\|_A\|S\|_A\neq 0$ and $\{x_n\}$ is a sequence of unit vectors in $\mathcal{H}$ satisfying \eqref{ThmB}, then it also satisfies
$$\lim\limits_{n\to \infty} \|Tx_n\|_A=\|T\|_A \qquad {\textrm{and}} \qquad \lim\limits_{n\to \infty} \|Sx_n\|_A=\|S\|_A.
$$
Indeed, for any $\epsilon >0$ and $n$  large enough we have
$$
\|T\|_A\|S\|_A\geq \|S\|_A\|Tx_n\|_A\geq | \langle Tx_n, Sx_n\rangle_A|\geq \|S\|_A\|T\|_A-\epsilon.
$$
 Hence, $\lim\limits_{n\to \infty} \|Tx_n\|_A=\|T\|_A.$  Analogously  by changing the roles between $T$ and $S$ we obtain $\lim\limits_{n\to \infty} \|Sx_n\|_A=\|S\|_A.$
\end{remark}
Now, we state the proof of Lemma \ref{lem3}.
\begin{proof}[Proof of Lemma \ref{lem3}]
Assume that $T\parallel_A S$, then by Theorem \ref{main1} there exists a sequence of $A$-unit vectors $\{x_n\}$ in $\mathcal{H}$ such that
\begin{equation}\label{zamanicana}
\displaystyle{\lim_{n\rightarrow +\infty}}|{\langle Tx_n, Sx_n\rangle}_{A}| = {\|T\|}_{A}\,{\|S\|}_{A}.
\end{equation}
Suppose that ${\|T\|}_{A}\,{\|S\|}_{A}\neq0$ (otherwise the desired assertion holds trivially). Since $\mathbb{T}$ is a compact subset of $\mathbb{C}$, then by taking a further subsequence we may assume that there is some $\lambda\in\mathbb{T}$ such that
$$\lim_{n\rightarrow\infty}\frac{\langle Tx_{n}, Sx_{n}\rangle_A}{|\langle Tx_{n}, Sx_{n}\rangle_A|}=\lambda.$$
So, by using \eqref{zamanicana} we get
$$\lim_{n\rightarrow\infty}\langle Tx_{n}, Sx_{n}\rangle_A=\lim_{n\rightarrow\infty}\frac{\langle Tx_{n}, Sx_{n}\rangle_A}{|\langle Tx_{n}, Sx_{n}\rangle_A|}|\langle Tx_{n}, Sx_{n}\rangle_A|=\lambda\|T\|_A\|S\|_A.$$
The converse implication follows immediately by applying Theorem \ref{main1}.
\end{proof}

In the following theorem we shall characterize the $A$-seminorm-parallelism of operators in $\mathcal{B}_A(\mathcal{H})$.

In what follows $\sigma(T)$, $\sigma_a(T)$, $r(T)$ and $W(T)$ stand for the spectrum, the approximate spectrum, the spectral radius and the numerical range of an arbitrary element $T\in \mathcal{B}(\mathcal{H})$, respectively.
\begin{lemma}(\cite[Theorem 1.2-1]{gu})\label{lem1}
Let $T\in \mathcal{B}(\mathcal{H})$. Then, $\sigma(T)\subseteq \overline{W(T)}$.
\end{lemma}

\begin{lemma}(\cite[Theorem 3.3.6]{mor})\label{state}
Let $T\in \mathcal{B}(\mathcal{H})$ be a normal operator. Then there exists a state $\psi$ $($i.e. a functional $\psi: \mathcal{B}(\mathcal{H}) \to \mathbb C$ with $\|\psi\|= 1$
and $\psi(T^* T) \geq 0$ for all $T \in \mathcal{B}(\mathcal H)$$)$ such that
$$\psi(T)=\|T\|.$$
\end{lemma}
Now, we are in a position to prove the following result.
\begin{theorem}\label{th.13}
Let $T, S\in \mathcal{B}_A(\mathcal{H})$. Then the following statements are equivalent:
\begin{itemize}
  \item [(1)] $T\parallel_A S$.
  \item [(2)] $r_A(S^{\sharp_A} T)=\|S^{\sharp_A} T\|_A=\|T^{\sharp_A} S\|_A=\|T\|_A\,\|S\|_A$.
  \item [(3)] $T^{\sharp_A} T\parallel_A T^{\sharp_A} S$ and $\|T^{\sharp_A} S\|_A=\|T\|_A\,\|S\|_A$.
\item [(4)]  $\|T^{\sharp_A}(T+\lambda S)\|_A=\|T\|_A(\|T\|_A+\|S\|_A)$ for some $\lambda\in\mathbb{T}$.
\end{itemize}
\end{theorem}
\begin{proof}
(1)$\Rightarrow$(2) Assume that $T\parallel_A S$. If $AT=0$ or $AS=0$, then by using \eqref{newsemi} we can verify that the assertion $(2)$ holds. Suppose that $AT\neq0$ and $AS\neq0$, i.e. ${\|T\|}_{A}\neq0$ and ${\|S\|}_{A}\neq0$. Since $T\parallel_A S$, then by Lemma \ref{lem3}, there exists a sequence of $A$-unit vectors $\{x_n\}$ in $\mathcal{H}$ satisfying
\begin{equation}\label{ii}
\lim_{n\rightarrow\infty} \langle Tx_n, Sx_n\rangle_A=\lambda\|T\|_A\,\|S\|_A,
\end{equation}
for some $\lambda\in\mathbb{T}$. This implies that
\begin{equation}\label{limreal}
\displaystyle{\lim_{n\rightarrow +\infty}}\Re\left({\langle Tx_n, \lambda Sx_n\rangle}_{A}\right)= {\|T\|}_{A}\,{\|S\|}_{A},
\end{equation}
where $\Re(z)$ denotes the real part of $z\in \mathbb{C}.$ Moreover, by using the Cauchy-Schwarz inequality it follows from
\begin{equation}\label{li0mreal}
{\|T\|}_{A}\,{\|S\|}_{A}=\displaystyle{\lim_{n\rightarrow +\infty}}|{\langle Tx_n, Sx_n\rangle}_{A}|\leq \displaystyle{\lim_{n\rightarrow +\infty}}\|Tx_n\|_A\,\|S\|_A\leq {\|T\|}_{A}\,{\|S\|}_{A}.
\end{equation}
 Then, \eqref{li0mreal} implies that $\displaystyle{\lim_{n\rightarrow +\infty}}\|Tx_n\|_A={\|T\|}_{A}$. In addition, by similar arguments as above, we obtain $\displaystyle{\lim_{n\rightarrow +\infty}}{\|Sx_n\|}_{A} = {\|S\|}_{A}$. So, by taking into consideration \eqref{limreal}, we see that
\begin{align*}
{\|T\|}_{A} +{\|S\|}_{A}
&\geq {\|T+\lambda S\|}_{A}\\
&\geq \left(\displaystyle{\lim_{n\rightarrow +\infty}}{\|(T+\lambda S)x_n\|}_{A}^2\right)^{1/2}\\
&\geq \left(\displaystyle{\lim_{n\rightarrow +\infty}}\left[{\|Tx_n\|}_{A}^2+2\Re\left({\langle Tx_n, \lambda Sx_n\rangle}_{A}\right)+{\|Sx_n\|}_{A}^2\right]\right)^{1/2}\\
& =\left({\|T\|}^2_{A} + 2{\|S\|}_{A}{\|T\|}_{A} + {\|S\|}^2_{A}\right)^{1/2}= {\|T\|}_{A} + {\|S\|}_{A}.
\end{align*}
Thus, we infer that ${\|T +\lambda S\|}_{A} = {\|T\|}_{A} + {\|S\|}_{A}$. Hence, it can be observed that
\begin{align*}
(\|T\|_A+\|S\|_A)^2
&=\|T+\lambda S\|_A^2\\
&=\|(T+\lambda S)^{\sharp_A}(T+\lambda S)\|_A\\
&\leq\|T^{\sharp_A} T\|_A+\|\lambda T^{\sharp_A} S\|_A+\|\overline{\lambda} S^{\sharp_A} T\|_A+\|S^{\sharp_A} S\|_A\\
&\leq\|T\|_A^2+2\|T\|_A\,\|S\|_A+\|S\|_A^2\\
&=(\|T\|_A+\|S\|_A)^2.
\end{align*}
This implies that $\|T^{\sharp_A} S\|_A+\|S^{\sharp_A} T\|_A=2\|T\|\,\|S\|.$ On the other hand, one observes that $P_{\overline{\mathcal{R}(A)}}A=AP_{\overline{\mathcal{R}(A)}}=A$. Moreover, by \eqref{newsemi}, we see that
\begin{align*}
\|T^{\sharp_A}S\|_A
&=\|S^{\sharp_A}P_{\overline{\mathcal{R}(A)}}TP_{\overline{\mathcal{R}(A)}}\|_A\\
&=\sup\left\{|\langle AP_{\overline{\mathcal{R}(A)}}x, (S^{\sharp_A}P_{\overline{\mathcal{R}(A)}}T)^{\sharp_A}y\rangle|\,;\;x,y\in \mathcal{H},\,\|x\|_{A}=\|y\|_{A}= 1\right\}\\
&=\sup\left\{|\langle S^{\sharp_A}P_{\overline{\mathcal{R}(A)}}Tx, y\rangle_A|\,;\;x,y\in \mathcal{H},\,\|x\|_{A}=\|y\|_{A}= 1\right\}\\
&=\sup\left\{|\langle AP_{\overline{\mathcal{R}(A)}}Tx, Sy\rangle|\,;\;x,y\in \mathcal{H},\,\|x\|_{A}=\|y\|_{A}= 1\right\}\\
&=\sup\left\{|\langle S^{\sharp_A}Tx, y\rangle_A|\,;\;x,y\in \mathcal{H},\,\|x\|_{A}=\|y\|_{A}= 1\right\}\\
&=\|S^{\sharp_A}T\|_A.
\end{align*}
Hence, we deduce that
\begin{equation}\label{125}
\|S^{\sharp_A}T\|_A=\|T^{\sharp_A} S\|_A=\|T\|_A\,\|S\|_A.
\end{equation}
Moreover, by using the Cauchy-Shwarz inequality, we see that
\begin{align*}
\|T\|_A\,\|S\|_A
&=\lim_{n\rightarrow\infty} |\langle Tx_n, Sx_n\rangle_A|\\
&\leq\lim_{n\rightarrow\infty} \|S^{\sharp_A} Tx_n\|_A\\
&\leq \|S^{\sharp_A} T\|_A =\|T\|_A\,\|S\|_A,
\end{align*}
where the last equality follows from \eqref{125}. So, we have
\begin{equation}\label{iii}
\lim_{n\rightarrow\infty} \|S^{\sharp_A} Tx_n\|_A=\|T\|_A\,\|S\|_A.
\end{equation}
 On the other hand, it can be observed that
\begin{align*}
\|(S^{\sharp_A} T-\lambda\|T\|_A\,\|S\|_AI)x_n\|_A^2
&=\|S^{\sharp_A} Tx_n\|_A^2-2\|T\|_A\,\|S\|_A\Re\left(\overline{\lambda}\langle Tx_n, Sx_n\rangle_A\right)\\
&+\|T\|_A^2\,\|S\|_A^2.
\end{align*}
So, by using \eqref{ii} together with \eqref{iii} we get
 $$\lim_{n\rightarrow\infty} \left\|\Big(S^{\sharp_A} T-\lambda\|T\|_A\,\|S\|_AI\Big)x_n\right\|_A=0.$$
This implies, thought \eqref{newhilbert}, that
 $$\lim_{n\rightarrow\infty} \left\|A\Big(S^{\sharp_A} T-\lambda\|T\|_A\,\|S\|_AI\Big)x_n\right\|_{\mathbf{R}(A^{1/2})}=0,$$
So, by using Lemma \ref{lem2} we get
 $$\lim_{n\rightarrow\infty} \left\|\Big((\widetilde{S})^{*} \widetilde{T}-\lambda\|T\|_A\,\|S\|_AI_{\mathbf{R}(A^{1/2})}\Big)Ax_n\right\|_{\mathbf{R}(A^{1/2})}=0.$$
 Since $\|Ax_n\|_{\mathbf{R}(A^{1/2})}=\|x_n\|_A=1$. Then, $\lambda\|T\|_A\,\|S\|_A\in \sigma_a\left((\widetilde{S})^{*} \widetilde{T}\right)$. So,
 $$\|T\|_A\,\|S\|_A\leq r\left((\widetilde{S})^{*} \widetilde{T}\right)=r(\widetilde{S^{\sharp_A} T})=r_A(S^{\sharp_A} T),$$
 where the last equality follows from Lemma \ref{lem2}. Further, clearly $r_A(S^{\sharp_A} T)\leq \|T\|_A\,\|S\|_A$. This proves, through \eqref{125}, that
\begin{equation*}
r_A(S^{\sharp_A} T)=\|T\|_A\,\|S\|_A=\|S^{\sharp_A} T\|_A=\|T^{\sharp_A} S\|_A,
\end{equation*}
as required.

$(2)\Rightarrow(1)$ Assume that (2) holds. Then, by applying Lemma \ref{lem2} we can see that
$$r\left((\widetilde{S})^{*} \widetilde{T}\right)=\|T\|_A\,\|S\|_A.$$
Hence, there exists $\lambda_0\in \sigma\left((\widetilde{S})^{*} \widetilde{T}\right)$ such that $|\lambda_0|=\|T\|_A\,\|S\|_A$. So, by Lemma \ref{lem1} together with Lemma \ref{lem2} we have
\begin{align*}
\lambda_0\in \overline{W\left((\widetilde{S})^{*} \widetilde{T}\right)}=\overline{W_A(S^{\sharp_A} T)}.
\end{align*}
Thus there exists a sequence of $A$-unit vectors $\{x_n\}$ in $\mathcal{H}$ satisfying
\begin{equation*}
\lim_{n\rightarrow\infty} \langle Tx_n, Sx_n\rangle_A=\lambda_0.
\end{equation*}
This immediately proves the desired result by applying Theorem \ref{main1}.\\
(1)$\Rightarrow$(3) Assume that $T\parallel_A S$. Then, by Lemma \ref{lem3} there exist a sequence of $A$-unit vectors $\{x_n\}$ in $\mathcal{H}$ and $\lambda\in\mathbb{T}$ such that
\begin{equation*}
\lim_{n\rightarrow\infty} \langle Tx_n, Sx_n\rangle_A=\lambda\|T\|_A\,\|S\|_A.
\end{equation*}
So by proceeding as in the implication (1)$\Rightarrow$(2), we obtain $\|T+\lambda S\|_A=\|T\|_A+\|S\|_A$ and $\|T^{\sharp_A} S\|_A=\|T\|_A\,\|S\|_A$. This implies, by Lemma \ref{lem2}, that
\begin{equation}\label{1651}
\|\widetilde{T}+\lambda \widetilde{S}\|_{\mathcal{B}(\mathbf{R}(A^{1/2}))}=\|\widetilde{T}\|_{\mathbf{R}(A^{1/2})}+\|\widetilde{S}\|_{\mathcal{B}(\mathbf{R}(A^{1/2}))}
\end{equation}
and
\begin{equation*}
\|(\widetilde{T})^{*} \widetilde{S}\|_{\mathcal{B}(\mathbf{R}(A^{1/2}))}=\|\widetilde{T}\|_{\mathcal{B}(\mathbf{R}(A^{1/2}))}\,\|\widetilde{S}\|_{\mathcal{B}(\mathbf{R}(A^{1/2}))}.
\end{equation*}
Since $(\widetilde{T}+\lambda \widetilde{S})^*(\widetilde{T}+\lambda \widetilde{S})$ is a normal operator on the Hilbert space $\mathbf{R}(A^{1/2})$ then by Lemma \ref{state}, there exists a state $\psi$ such that
such that
\begin{align*}
\psi\left((\widetilde{T}+\lambda \widetilde{S})^*(\widetilde{T}+\lambda \widetilde{S})\right)
&=\|(\widetilde{T}+\lambda \widetilde{S})^*(\widetilde{T}+\lambda \widetilde{S})\|_{\mathcal{B}(\mathbf{R}(A^{1/2}))}\\
&=\|\widetilde{T}+\lambda \widetilde{S}\|_{\mathcal{B}(\mathbf{R}(A^{1/2}))}^2=\left(\|\widetilde{T}\|_{\mathcal{B}(\mathbf{R}(A^{1/2}))}+\|\widetilde{S}\|_{\mathcal{B}(\mathbf{R}(A^{1/2}))}\right)^2,
\end{align*}
where the last equality follows from \eqref{1651}. Thus

\begin{align*}
&\left(\|\widetilde{T}\|_{\mathcal{B}(\mathbf{R}(A^{1/2}))}+\|\widetilde{S}\|_{\mathcal{B}(\mathbf{R}(A^{1/2}))}\right)^2\\
&=\psi\left((\widetilde{T})^{*} \widetilde{T}+\lambda (\widetilde{T})^{*} \widetilde{S}+\overline{\lambda}(\widetilde{S})^* \widetilde{T}+(\widetilde{S})^* \widetilde{S}\right)\\
&\leq\|(\widetilde{T})^{*} \widetilde{T}\|_{\mathcal{B}(\mathbf{R}(A^{1/2}))}+\|\lambda (\widetilde{T})^{*} \widetilde{S}+\overline{\lambda}(\widetilde{S})^* \widetilde{T}\|_{\mathcal{B}(\mathbf{R}(A^{1/2}))}+\|(\widetilde{S})^* \widetilde{S}\|_{\mathcal{B}(\mathbf{R}(A^{1/2}))}\\
&\leq\|\widetilde{T}\|_{\mathcal{B}(\mathbf{R}(A^{1/2}))}^2+2\|\widetilde{T}\|_{\mathcal{B}(\mathbf{R}(A^{1/2}))}\,\|\widetilde{S}\|_{\mathcal{B}(\mathbf{R}(A^{1/2}))}+\|\widetilde{S}\|_{\mathcal{B}(\mathbf{R}(A^{1/2}))}^2
\\
&=\left(\|\widetilde{T}\|_{\mathcal{B}(\mathbf{R}(A^{1/2}))}+\|\widetilde{S}\|_{\mathcal{B}(\mathbf{R}(A^{1/2}))}\right)^2.
\end{align*}
Hence $\psi\left((\widetilde{T})^{*} \widetilde{T}\right)=\|(\widetilde{T})^{*} \widetilde{T}\|_{\mathcal{B}(\mathbf{R}(A^{1/2}))}$ and $\psi\left(\lambda (\widetilde{T})^{*} \widetilde{S}\right)=\|(\widetilde{T})^{*} \widetilde{S}\|_{\mathcal{B}(\mathbf{R}(A^{1/2}))}$. Therefore
\begin{align*}
\|(\widetilde{T})^{*} \widetilde{T}\|_{\mathcal{B}(\mathbf{R}(A^{1/2}))}+\|(\widetilde{T})^{*} \widetilde{S}\|_{\mathcal{B}(\mathbf{R}(A^{1/2}))}
& =\psi\left((\widetilde{T})^{*} \widetilde{T}+\lambda (\widetilde{T})^{*} \widetilde{S}\right) \\
 &\leq\|(\widetilde{T})^{*} \widetilde{T}+\lambda (\widetilde{T})^{*} \widetilde{S}\|_{\mathcal{B}(\mathbf{R}(A^{1/2}))}\\
 & \leq\|(\widetilde{T})^{*} \widetilde{T}\|_{\mathcal{B}(\mathbf{R}(A^{1/2}))}+\|(\widetilde{T})^{*} \widetilde{S}\|_{\mathcal{B}(\mathbf{R}(A^{1/2}))}.
\end{align*}
So, we deduce that
 $$\|(\widetilde{T})^{*} \widetilde{T}+\lambda (\widetilde{T})^{*} \widetilde{S}\|_{\mathcal{B}(\mathbf{R}(A^{1/2}))}=\|(\widetilde{T})^{*} \widetilde{T}\|_{\mathcal{B}(\mathbf{R}(A^{1/2}))}+\|(\widetilde{T})^{*} \widetilde{S}\|_{\mathcal{B}(\mathbf{R}(A^{1/2}))},$$
  for some $\lambda\in\mathbb{T}$. Thus $(\widetilde{T})^{*} \widetilde{T}\parallel (\widetilde{T})^{*} \widetilde{S}$ which implies that $\widetilde{T^{\sharp_A} T}\parallel \widetilde{T^{\sharp_A} S}$. So, by Lemma \ref{lem2}(v), $T^{\sharp_A} T\parallel_A T^{\sharp_A} S$.\\
(3)$\Rightarrow$(4) Follows obviously.\\
(4)$\Rightarrow$(1) Assume that $\|T^{\sharp_A}(T+\lambda S)\|_A=\|T\|_A(\|T\|_A+\|S\|_A)$ for some $\lambda\in\mathbb{T}.$ Then we see that
\begin{align*}
\|T\|_A(\|T\|_A+\|S\|_A)
&\geq\|T^{\sharp_A}\|_A\|T+\lambda S\|_A\\
&\geq\|T^{\sharp_A}(T+\lambda S)\|_A\\
&=\|T\|_A(\|T\|_A+\|S\|_A).
\end{align*}
So, if $AT\neq 0$, then $\|T+\lambda S\|_A=\|T\|_A+\|S\|_A$ which yields that $T\parallel_A S$.  Moreover, if $AT=0$, then by taking into account \eqref{newsemi} we prove that $T\parallel_A S$.
\end{proof}

\begin{corollary}\label{corn_1}
Let $T, S \in \mathcal{B}_{A}(\mathcal{H})$. The following conditions are equivalent:
\begin{itemize}
  \item [(1)] $T\parallel_A S.$
  \item [(2)] $\omega_A(S^{\sharp_A} T)=\|S^{\sharp_A} T\|_A=\|T^{\sharp_A} S\|_A=\|T\|_A\,\|S\|_A$.
\end{itemize}
\end{corollary}
To prove Corollary \ref{corn_1} we need the following Lemma.

\begin{lemqt}\label{normaloid}
Let $T\in \mathcal{B}_{A}(\mathcal{H})$. Then $T$ is $A$-normaloid if and only if $\omega_A(T)=\|T\|_A$.
\end{lemqt}
Now, we state the proof of Corollary \ref{corn_1}.
\begin{proof}[Proof of Corollary \ref{corn_1}]
(1)$\Rightarrow$(2)  Assume that $T\parallel_A S$. Then, by Theorem
\ref{th.13} we have  $r_A(S^{\sharp_A} T)=\|S^{\sharp_A} T\|_A=\|T^{\sharp_A} S\|_A=\|T\|_A\,\|S\|_A$. In particular, $S^{\sharp_A} T$ is $A$-normaloid. So, by Lemma \ref{normaloid}, $\omega_A(S^{\sharp_A} T)=\|S^{\sharp_A} T\|_A$.

(2)$\Rightarrow$(1) Assume that $\omega_A(S^{\sharp_A} T)=\|S^{\sharp_A} T\|_A=\|T^{\sharp_A} S\|_A=\|T\|_A\,\|S\|_A$. In particular, by Lemma \ref{normaloid}, we conclude that $S^{\sharp_A} T$ is $A$-normaloid. So, by \cite[Proposition 4]{feki01} there exists a sequence of $A$-unit vectors $\{x_n\}$ such that $$\lim_{n\rightarrow +\infty}{\|S^{\sharp_A} Tx_n\|}_A = {\|S^{\sharp_A} T\|}_A\text{ and } \displaystyle{\lim_{n\rightarrow +\infty}}|{\langle S^{\sharp_A}Tx_n, x_n\rangle}_A |= \omega_A(S^{\sharp_A} T).$$
This implies that
\begin{equation*}
\displaystyle{\lim_{n\rightarrow +\infty}}|{\langle Tx_n, Sx_n\rangle}_A |=\|T\|_A\,\|S\|_A.
\end{equation*}
Thus, by Theorem \ref{main1} we conclude that $T\parallel_A S$.
\end{proof}

Next, we investigate the case when an operator $T\in \mathcal{B}_A(\mathcal{H})$ is $A$-parallel to the identity operator.
\begin{theorem}\label{th.19}
Let $T\in \mathcal{B}_A(\mathcal{H})$. Then the following statements are equivalent:
\begin{itemize}
  \item [(1)] $T\parallel_A I$.
  \item [(2)] $T\parallel_A T^{\sharp_A}$.
  \item [(3)] $T^{\sharp_A} T\parallel_A T^{\sharp_A}$.
\end{itemize}
\end{theorem}
\begin{proof}
(1)$\Leftrightarrow$(2) Assume that $T\parallel_A I$. Then, by Lemma \ref{lem2} $(v)$, $\widetilde{T}\parallel I_{\mathbf{R}(A^{1/2})}$. So, $\|\widetilde{T}+\lambda I_{\mathbf{R}(A^{1/2})}\|_{\mathcal{B}(\mathbf{R}(A^{1/2}))}=\|\widetilde{T}\|_{\mathcal{B}(\mathbf{R}(A^{1/2}))}+1$, for some $\lambda\in\mathbb{T}$. Then by Lemma \ref{state} there exists a state $\psi$ such that
such that
\begin{align*}
\psi\left((\widetilde{T}+\lambda I_{\mathbf{R}(A^{1/2})})^*(\widetilde{T}+\lambda I_{\mathbf{R}(A^{1/2})})\right)
&=\|(\widetilde{T}+\lambda I_{\mathbf{R}(A^{1/2})})^*(\widetilde{T}+\lambda I_{\mathbf{R}(A^{1/2})})\|_{\mathcal{B}(\mathbf{R}(A^{1/2}))}\\
&=\|\widetilde{T}+\lambda I_{\mathbf{R}(A^{1/2})}\|_{\mathcal{B}(\mathbf{R}(A^{1/2}))}^2\\
&=\left(\|\widetilde{T}\|_{\mathcal{B}(\mathbf{R}(A^{1/2}))}+1\right)^2.
\end{align*}
So, we see that
\begin{align*}
\left(\|\widetilde{T}\|_{\mathcal{B}(\mathbf{R}(A^{1/2}))}+1\right)^2
&=\psi\left((\widetilde{T}+\lambda I_{\mathbf{R}(A^{1/2})})(\widetilde{T}+\lambda I_{\mathbf{R}(A^{1/2})})^*\right)\\
&=\psi\left(\widetilde{T}(\widetilde{T})^*\right)+\psi\big(\overline{\lambda}\widetilde{T}\big)+\psi\left(\lambda (\widetilde{T})^*\right)+1\\
&\leq\|\widetilde{T}(\widetilde{T})^*\|_{\mathcal{B}(\mathbf{R}(A^{1/2}))}+\|\overline{\lambda}\widetilde{T}\|_{\mathcal{B}(\mathbf{R}(A^{1/2}))}+\|\lambda (\widetilde{T})^*\|_{\mathcal{B}(\mathbf{R}(A^{1/2}))}+1\\
&=\|\widetilde{T}\|_{\mathcal{B}(\mathbf{R}(A^{1/2}))}^2+2\|\widetilde{T}\|_{\mathcal{B}(\mathbf{R}(A^{1/2}))}+1\\
&=\left(\|\widetilde{T}\|_{\mathcal{B}(\mathbf{R}(A^{1/2}))}+1\right)^2.
\end{align*}
Therefore $\psi(\overline{\lambda}\widetilde{T})=\psi\left(\lambda (\widetilde{T})^*\right)=\|\widetilde{T}\|_{\mathcal{B}(\mathbf{R}(A^{1/2}))}$. This yields that
\begin{align*}
\|\widetilde{T}\|_{\mathcal{B}(\mathbf{R}(A^{1/2}))}+\|(\widetilde{T})^*\|_{\mathcal{B}(\mathbf{R}(A^{1/2}))}
&=\psi\left(\overline{\lambda}\widetilde{T}+\lambda (\widetilde{T})^*\right)\\
&\leq\|\overline{\lambda}\widetilde{T}+\lambda (\widetilde{T})^*\|\\
&=\|\widetilde{T}+\lambda^2 (\widetilde{T})^*\|_{\mathcal{B}(\mathbf{R}(A^{1/2}))}\\
&\leq\|\widetilde{T}\|_{\mathcal{B}(\mathbf{R}(A^{1/2}))}+\|(\widetilde{T})^*\|_{\mathcal{B}(\mathbf{R}(A^{1/2}))}.
\end{align*}
Hence,
$$\|\widetilde{T}+\lambda^2 (\widetilde{T})^*\|_{\mathcal{B}(\mathbf{R}(A^{1/2}))}=\|\widetilde{T}\|_{\mathcal{B}(\mathbf{R}(A^{1/2}))}+\|(\widetilde{T})^*\|_{\mathcal{B}(\mathbf{R}(A^{1/2}))},$$ in which $\lambda^2\in\mathbb{T}$. So $\widetilde{T}\parallel_A (\widetilde{T})^*$. This implies, by Lemma \ref{lem2}, that $\widetilde{T}\parallel_A \widetilde{T^{\sharp_A}}$ which in turn yields that $T\parallel_A T^{\sharp_A}$.

Conversely, assume that $T\parallel_A T^{\sharp_A}$ this implies, by Lemma \ref{lem2}, that $\widetilde{T}\parallel (\widetilde{T})^*$ which, in turn, yields that
 $$\|\widetilde{T}+\lambda (\widetilde{T})^*\|_{\mathcal{B}(\mathbf{R}(A^{1/2}))}=2\|\widetilde{T}\|_{\mathcal{B}(\mathbf{R}(A^{1/2}))},$$
  for some $\lambda\in\mathbb{T}$. Since $\widetilde{T}+\lambda (\widetilde{T})^*$ is a normal operator on the Hilbert space $\mathbf{R}(A^{1/2})$, then by Lemma \ref{state}, there exists a state $\psi$ such that
 $$\left|\psi\left(\widetilde{T}+\lambda (\widetilde{T})^*\right)\right|=\|\widetilde{T}+\lambda (\widetilde{T})^*\|_{\mathcal{B}(\mathbf{R}(A^{1/2}))}=2\|\widetilde{T}\|_{\mathcal{B}(\mathbf{R}(A^{1/2}))}.$$
Hence, we obtain
$$2\|\widetilde{T}\|_{\mathcal{B}(\mathbf{R}(A^{1/2}))}=\left|\psi\left(\widetilde{T}+\lambda (\widetilde{T})^*\right)\right|\leq2|\psi(\widetilde{T})|\leq2\|\widetilde{T}\|_{\mathcal{B}(\mathbf{R}(A^{1/2}))}.$$
This implies that $|\psi(\widetilde{T})|=\|\widetilde{T}\|_{\mathcal{B}(\mathbf{R}(A^{1/2}))}$. So, there exists a number $\delta\in\mathbb{T}$ such that $\psi(\widetilde{T})=\delta\|\widetilde{T}\|_{\mathcal{B}(\mathbf{R}(A^{1/2}))}$. Thus, we deduce that
\begin{align*}
\|\widetilde{T}\|_{\mathcal{B}(\mathbf{R}(A^{1/2}))}+1
&=\psi\big(\overline{\delta}\widetilde{T}+I_{\mathbf{R}(A^{1/2})}\big)\\
&\leq\|\overline{\delta}\widetilde{T}+I_{\mathbf{R}(A^{1/2})}\|_{\mathcal{B}(\mathbf{R}(A^{1/2}))}\\
&=\|\widetilde{T}+\delta I_{\mathbf{R}(A^{1/2})}\|_{\mathcal{B}(\mathbf{R}(A^{1/2}))}\leq\|\widetilde{T}\|_{\mathcal{B}(\mathbf{R}(A^{1/2}))}+1.
\end{align*}
So $\|\widetilde{T}+\delta I_{\mathbf{R}(A^{1/2})}\|_{\mathcal{B}(\mathbf{R}(A^{1/2}))}=\|\widetilde{T}\|_{\mathcal{B}(\mathbf{R}(A^{1/2}))}+1$ which implies that $\widetilde{T}\parallel_A I_{\mathbf{R}(A^{1/2})}$. Hence, $T\parallel_A I$ as required.

(1)$\Leftrightarrow$(3) Follows from Theorem \ref{th.13}.

\end{proof}

In the next two theorems, we give some characterizations when the $A$-Davis Wielandt radius of semi-Hilbert space operators attains its upper bound for operators in $\mathcal{B}_{A^{1/2}}(\mathcal{H})$ and $\mathcal{B}_{A}(\mathcal{H})$, respectively.

\begin{theorem}\label{equinew2}
Let $T\in \mathcal{B}_{A^{1/2}}(\mathcal{H})$. Then, the following assertions are equivalent:
\begin{itemize}
\item [(1)] $d\omega_A(T)=\sqrt{\omega_A(T)^2+\|T\|_A^4}$.
\item [(2)] $T\parallel_A I$.
\item [(3)] $T$ is $A$-normaloid.
\item [(4)] $\omega_A^2(T)A\geq T^*AT.$
\end{itemize}
\end{theorem}
\begin{proof}
The equivalences $(1)\Leftrightarrow(2)$ and $(2)\Leftrightarrow(3)$ have been proved in \cite{fekisidha2019}.\\
$(3)\Leftrightarrow(4):$ By Lemma \ref{normaloid}, $T$ is $A$-normaloid if and only if $\omega_A(T)=\|T\|_A$. On the other hand, it be observed that
\begin{align*}
\omega_A(T)=\|T\|_A
&\Leftrightarrow  \|Tx\|_A\leq \omega_A(T)\|x\|_A,\;\forall\,x\in \mathcal{H}\\
 &\Leftrightarrow  \|Tx\|_A^2\leq \omega_A(T)^2\|x\|_A^2,\;\forall\,x\in \mathcal{H}\\
  &\Leftrightarrow  \langle T^*ATx, x \rangle_A\leq \langle \omega_A(T)^2x, x \rangle_A,\;\forall\,x\in \mathcal{H}\\
    &\Leftrightarrow  \langle (T^*AT-\omega_A(T)^2A)x, x \rangle\leq 0,\;\forall\,x\in \mathcal{H}\\
&\Leftrightarrow \omega_A^2(T)A\geq T^*AT.
\end{align*}
This achieves the proof.
\end{proof}

\begin{theorem}
Let $T\in \mathcal{B}_A(\mathcal{H})$. The following statements are equivalent:
\begin{itemize}
\item[(1)] $d\omega_A(T) = \sqrt{\omega_A^2(T) + \|T\|_A^4}$.
\item[(2)] There exists a sequence of $A$-unit vectors $\{x_n\}$ in $\mathcal{H}$ such that
\begin{equation*}
\lim_{n\rightarrow\infty} \big|\langle T^2x_n, x_n\rangle_A\big| = \|T\|_A^2.
\end{equation*}
\item[(3)] There exists a sequence of $A$-unit vectors $\{x_n\}$ in $\mathcal{H}$ such that
\begin{equation*}
\lim_{n\rightarrow\infty} \big|\langle TT^{\sharp_A} Tx_n, x_n\rangle_A\big| = \|T\|_A^3.
\end{equation*}
\item[(4)] $\omega_A(T^2)=\|T\|_A^2$.
\end{itemize}
\end{theorem}
\begin{proof}
$(1)\Leftrightarrow(2):$ By Theorem \ref{equinew2}, we have $d\omega_A(T) = \sqrt{\omega_A^2(T) + \|T\|_A^4}$ if and only if $T\parallel_A I$ which in turn equivalent, by Theorem \ref{th.19}, to $T\parallel_A T^{\sharp_A}.$ On the other hand, in view of Theorem \ref{main1}, we have $T\parallel_A T^{\sharp_A}$ if and only if there exists a sequence of $A$-unit vectors $\{x_n\}$ in $\mathcal{H}$ such that
    \begin{equation*}
    \lim_{n\to \infty}|\langle T x_n, T^{\sharp_A} x_n\rangle_A|=\|T\|_A\|T^{\sharp_A}\|_A.
    \end{equation*}
So, we reach the equivalence $(1)\Leftrightarrow(2)$ since $\|T\|_A=\|T^{\sharp_A}\|_A$.

$(1)\Leftrightarrow(3):$ By proceeding as above and taking into consideration Theorem \ref{th.19}, we deduce that $d\omega_A(T) = \sqrt{\omega_A^2(T) + \|T\|_A^4}$ if and only if $T^{\sharp_A} T\parallel_A T^{\sharp_A}$ which is in turn equivalent, by Theorem \ref{th.19}, to the existence of a sequence of $A$-unit vectors $\{x_n\}$ in $\mathcal{H}$ such that
    \begin{equation*}
    \lim_{n\to \infty}|\langle T^{\sharp_A} T x_n, T^{\sharp_A} x_n\rangle_A|=\|T^{\sharp_A} T\|_A\|T^{\sharp_A}\|_A.
    \end{equation*}
Thus, we obtain the desired equivalence since $\|T\|_A^2=\|T^{\sharp_A} T\|_A$ and $$|\langle T^{\sharp_A} T x_n, T^{\sharp_A} x_n\rangle_A|=|\langle TT^{\sharp_A} T x_n, x_n\rangle_A|.$$

$(1)\Leftrightarrow(4):$ If $d\omega_A(T) = \sqrt{\omega_A^2(T) + \|T\|_A^4}$, then by Theorem \ref{equinew2} $T$ is $A$-normaloid. This implies that $T$ is $A$-spectraloid. So, by \cite[Theorem 6]{feki01} $\omega_A(T^2)=\omega_A^2(T)$. Conversely, assume that $\omega_A(T^2)=\|T\|_A^2$. This implies that the assertion $(2)$ holds and so $(1)$ holds.
\end{proof}

For $x, y\in \mathcal{H}$, the $A$-rank one operators is defined in \cite{bhunfekipaul} by
\begin{align*}
x\otimes_A y\colon \mathcal{H} & \rightarrow\mathcal{H},\;z\mapsto (x\otimes_A y)(z):=\langle z, y\rangle_Ax.
\end{align*}
A characterization of the $A$-parallelism of $x\otimes_A y$ and the identity operator is stated as follows.
\begin{corollary}\label{cortensor}
Let $x, y \in \mathcal{H}$, the following conditions are equivalent:
\begin{itemize}
\item[(1)] $x\otimes_A y \parallel_A I$.
\item[(2)] $d\omega_A(x\otimes_A y) = \sqrt{\omega_A^2(x\otimes_A y) + \|x\otimes_A y\|_A^4}$.
\item[(3)] The vectors $A^{1/2}x$ and $A^{1/2}y$ are linearly dependent.
\item[(4)] The vectors $Ax$ and $Ay$ are linearly dependent.
\end{itemize}
\end{corollary}
To prove Corollary \ref{cortensor} we need the following lemma.
\begin{lemqt}\label{lemonerank}(\cite{zamani3})
Let $x, y\in \mathcal{H}$. Then, the following statement hold:
\begin{itemize}
  \item [(i)] $\|x\otimes_A y\|_A=\|x\|_A\|y\|_A$.
  \item [(ii)] $\omega_A(x\otimes_Ay) = \frac{1}{2}\left(|\langle x, y\rangle_A| + \| x\|_A \| y\|_A\right)$.
\end{itemize}
\end{lemqt}
Now we are ready to prove Corollary \ref{cortensor}.
\begin{proof}[Proof of Corollary \ref{cortensor}]
$(1)\Leftrightarrow(2):$ Follows immediately from Theorem \ref{equinew2}.

$(2)\Leftrightarrow(3):$ By the equivalence (2)$\Leftrightarrow$(3) of Theorem \ref{equinew2} we infer that
$$d\omega_A(x\otimes_A y) = \sqrt{\omega_A^2(x\otimes_A y) + \|x\otimes_A y\|_A^4}\Leftrightarrow \omega_A(x\otimes_A y)=\|x\otimes_A y\|_A.$$
Moreover, by using Lemma \ref{lemonerank}, we see that
\begin{align*}
\omega_A(x\otimes_A y)=\|x\otimes_A y\|_A
&\Leftrightarrow \tfrac{1}{2}\left(|\langle x, y\rangle_A| + \| x\|_A \| y\|_A\right)=\| x\|_A \| y\|_A\\
&\Leftrightarrow |\langle x, y\rangle_A|=\| x\|_A \| y\|_A
\end{align*}
On the other hand $|\langle x, y\rangle_A|=\| x\|_A \| y\|_A$ if and only if the vectors $A^{1/2}x$ and $A^{1/2}y$ are linearly dependent.

$(3)\Leftrightarrow(4):$ This equivalence follows immediately since $\mathcal{N}(A)=\mathcal{N}(A^{1/2})$.
Hence, the proof is complete.
\end{proof}

\section{Further characterizations of $A$-seminorm-parallelism}\label{s3}
Our aim in this section is to give further characterizations of $A$-seminorm-parallelism via $A$-Birkhoff-James orthogonality of $A$-bounded operators. Recall also from \cite{zamani2019_2} that an element $T \in \mathcal{B}_{A^{1/2}}(\mathcal{H})$ is said to be $A$-Birkhoff-James orthogonal
to another element $S \in \mathcal{B}_{A^{1/2}}(\mathcal{H})$, denoted by $T\perp^{BJ}_A S$, if
\begin{align*}
{\|T + \gamma S\|}_A\geq {\|T\|}_A \quad \mbox{for all} \,\, \gamma \in \mathbb{C}.
\end{align*}



\begin{theorem}\label{Dau}
Let $T, S\in \mathcal{B}_{A^{1/2}}(\mathcal{H})$, then the following conditions are equivalent:
\begin{itemize}
\item [(1)] $T\parallel_A S$.
\item [(2)] $T \perp _{BJ} \|S\|_A T-\lambda \|T\|_AS,$ for some $\lambda \in \mathbb{T}.$
\item [(3)] $S \perp _A^{BJ} \lambda\|T\|_A S- \|S\|_A T,$ for some $\lambda \in \mathbb{T}.$
\end{itemize}
In addition if $\|T\|_A\|S\|_A\neq 0$ then $(1)$ to $(3)$ are also equivalent to
\begin{itemize}
\item [(4)] There exist a sequence of $A$-unit vectors $\{x_n\}$ in $\mathcal{H}$ and $\lambda \in \mathbb{T}$ such that
$$\lim\limits_{n\to \infty} \|Sx_n\|_A=\|S\|_A\;\text{ and }\; \lim\limits_{n\to \infty}\left\Vert Tx_n-\lambda\frac{\|T\|_A}{\|S\|_A}Sx_n\right\Vert_A=0.$$
\item [(5)] There exist a sequence of $A$-unit vectors $\{x_n\}$ in $\mathcal{H}$ and $\lambda \in \mathbb{T}$ such that
$$\lim\limits_{n\to \infty} \|Tx_n\|_A=\|T\|_A\;\text{ and }\;\lim\limits_{n\to \infty}\left\Vert Sx_n-\lambda\frac{\|S\|_A}{\|T\|_A}Tx_n\right\Vert_A=0.$$
\end{itemize}
\end{theorem}
In order to prove Theorem \ref{Dau} we need to recall from \cite{zamani2019_2} the following result.
\begin{thqt} \label{zamaniABJ}(\cite{zamani2019_2})
Let $T,S\in \mathcal{B}_{A^{1/2}}(\mathcal{H})$. Then, $T\perp^{BJ}_A S$
if and only if there exists a sequence of $A$-unit vectors $\{x_n\}$ in $\mathcal{H}$ such that
\begin{align*}
\displaystyle{\lim_{n\rightarrow +\infty}}{\|Tx_n\|}_A = {\|T\|}_A
\quad \mbox{and} \quad \displaystyle{\lim_{n\rightarrow +\infty}}{\langle Tx_n, Sx_n\rangle}_A = 0.
\end{align*}
\end{thqt}
Now we are ready to prove Theorem \ref{Dau}.
\begin{proof}[Proof of Theorem \ref{Dau}]
$(1)\Leftrightarrow(2):$ Assume that $T\parallel_A S$. If $\|S\|_A=0$, then by using \eqref{newsemi} it can be seen that the assertion $(2)$ holds. Now, suppose that $\|S\|_A\neq0$. Since $T\parallel_A S$, then by Lemma \ref{lem3} there exist a sequence of $A$-unit vectors $\{x_n\}$ in $\mathcal{H}$ and $\lambda\in \mathbb{T}$ such that
 $$
 \lim\limits_{n\to \infty}  \langle Tx_n, Sx_n\rangle_A=\lambda\|T\|_A\|S\|_A.$$ So, by Remark \ref{remparallelism} $\lim\limits_{n\to \infty}\|Tx_n\|_A=\|T\|_A$. Furthermore, we see that
\begin{align*}
\lim_{n\rightarrow\infty}  \langle Tx_n,  (\|S\|_A T-\lambda \|T\|_AS)x_n\rangle_A
&=\lim_{n\rightarrow\infty}  \|S\|_A \|Tx_n\|_A^2 -\overline{\lambda} \|T\|_A \langle Tx_n, Sx_n\rangle_A\nonumber\\
&=\|S\|_A\|T\|_A^2-\|T\|_A^2\|S\|_A=0.
\end{align*}
Thus, in view of Theorem \ref{zamaniABJ}, the second assertion holds. Conversely, assume $T \perp _A^{BJ}  \|S\|_A T-\lambda \|T\|_AS,$ for some $\lambda \in \mathbb{T}.$ If $\|T\|_A= 0$, then obviously $T\parallel_AS$. Suppose that $\|T\|_A\neq0$. By Theorem \ref{zamaniABJ}, there exists a sequence of $A$-unit vectors $\{y_n\}$ in $\mathcal{H}$ such that
 \begin{equation*}
 \lim_{n\to \infty} \|Ty_n\|_A=\|T\|_A\quad \text{and}\quad \lim_{n\rightarrow\infty}  \langle Ty_n,  (\|S\|_A T-\lambda \|T\|_AS)y_n\rangle_A=0.
 \end{equation*}
 Then, we deduce that
 $$
 \lim_{n\to \infty} \langle Ty_n, Sy_n\rangle_A=\frac{\lambda}{\|T\|_A}\lim_{n\to \infty} \|S\|_A\|Ty_n\|_A^2=\lambda\|T\|_A\|S\|_A.
 $$
$(1)\Leftrightarrow(3):$ The proof is analogous to the previous equivalence by changing the roles between $T$ and $S$.\\
$(1)\Leftrightarrow(4):$ By Lemma \ref{lem3} and Remark \ref{remparallelism}, there exist a sequence of $A$-unit vectors $\{x_n\}$ in $\mathcal{H}$ and $\lambda\in \mathbb{T}$ such that
$\lim\limits_{n\to \infty} \langle Tx_n, Sx_n\rangle=\lambda \|T\|_A\|S\|_A$,
  $\lim\limits_{n\to \infty}\|Tx_n\|_A=\|T\|_A$ and $\lim\limits_{n\to \infty}\|Sx_n\|_A=\|S\|_A.$ Thus
\begin{align*}
&\left\Vert Tx_n-\lambda\frac{\|T\|_A}{\|S\|_A}Sx_n\right\Vert_A^2\\
&=  \|Tx_n\|_A^2-\overline{\lambda}\frac{\|T\|_A}{\|S\|_A}\langle Tx_n, Sx_n\rangle_A  -\lambda\frac{\|T\|_A}{\|S\|_A}\langle Sx_n, Tx_n\rangle_A+\frac{\|T\|_A^2}{\|S\|_A^2}\|Sx_n\|_A^2,
\end{align*}
and so $\lim\limits_{n\to \infty}\left\Vert Tx_n-\lambda\frac{\|T\|_A}{\|S\|_A}Sx_n\right\Vert_A^2=0.$\\
Conversely, suppose that (4) is holds. Then
\begin{align*}
\|S\|_A+\|T\|_A
&\geq \|T+\lambda S\|_A\geq \|Tx_n+\lambda Sx_n\|_A\nonumber\\
&=\left\Vert(Tx_n-\lambda\frac{\|T\|_A}{\|S\|_A}Sx_n)-(-\lambda Sx_n-\lambda\frac{\|T\|_A}{\|S\|_A}Sx_n)\right\Vert_A\nonumber \\
&\geq \left\Vert\lambda Sx_n+\lambda\frac{\|T\|_A}{\|S\|_A}Sx_n\right\Vert_A - \left\Vert Tx_n-\lambda\frac{\|T\|_A}{\|S\|_A}Sx_n\right\Vert_A \nonumber\\
&= (\|S\|_A+\|T\|_A)\frac{\|Sx_n\|_A}{\|S\|_A}- \left\Vert Tx_n-\lambda\frac{\|T\|_A}{\|S\|_A}Sx_n\right\Vert_A.
\end{align*}
By taking limits, we get $\|S\|_A+\|T\|_A=\|T+\lambda S\|_A$. Then $T\parallel_A S$.\\
$(1)\Leftrightarrow(5):$ The proof is analogous to the previous equivalence by changing the roles between $T$ and $S$.
\end{proof}

\begin{corollary}\label{cor22}
Let $T\in \mathcal{B}_A(\mathcal{H})$. Then the following statements are equivalent:
\begin{itemize}
  \item [(1)] $T\parallel_A I$.
  \item [(2)] $T^p\parallel_A I$ for every $p\in\mathbb{N}$.
  \item [(3)] $T^p\parallel_A ({T^{\sharp_A}})^p$ for every $p\in\mathbb{N}$.
\end{itemize}
\end{corollary}
\begin{proof}
(1)$\Rightarrow$(2) Assume that $T\parallel_A I$. Then, by Theorem \ref{Dau}, there exists a sequence of $A$-unit vectors $\{x_n\}$ in $\mathcal{H}$ and $\lambda\in\mathbb{T}$ such that
$$\lim_{n\rightarrow\infty} \Big\|Tx_n-\lambda\|T\|_Ax_n\Big\|_A=0.$$
For every $i\in\mathbb{N}$ we have
\begin{align*}
&\Big\|\left(T^{i+1}-\lambda^{i+1}\|T\|_A^{i+1}I\right)x_n\Big\|_A\\
&=\Big\|T\left(T^{i}-\lambda^{i}\|T\|_A^{i}I\right)x_n+\lambda^{i}\|T\|_A^{i}\left(T-\lambda\|T\|_AI\right)x_n\Big\|_A\\
&\leq \|T\|_A\,\Big\|(T^{i}-\lambda^{i}\|T\|_A^{i}I)x_n\Big\|_A+\|T\|_A^i\,\Big\|(T-\lambda\|T\|_AI)x_n\Big\|_A.
\end{align*}
So, by induction, it can be shown that for every $p\in\mathbb{N}$ we have
\begin{equation}\label{cc2}
\lim_{n\rightarrow\infty} \Big\|(T^p-\lambda^p\|T\|_A^pI)x_n\Big\|_A=0.
\end{equation}
This implies, by Lemma \ref{lem2}, that
$$\lim_{n\rightarrow\infty} \Big\|\left((\widetilde{T})^p-\lambda^p\|T\|_A^pI_{\mathbf{R}(A^{1/2})}\right)Ax_n\Big\|_{\mathcal{B}(\mathbf{R}(A^{1/2}))}=0,$$
for every $p\in\mathbb{N}$. Hence, $\lambda^p\|T\|_A^p\in \sigma_a\left( (\widetilde{T})^p\right)$. So, we obtain
$$\|\widetilde{T}\|_{\mathcal{B}(\mathbf{R}(A^{1/2}))}^p\leq r\left((\widetilde{T})^p\right)\leq\left\|(\widetilde{T})^p\right\|_{\mathcal{B}(\mathbf{R}(A^{1/2}))}\leq\left\|\widetilde{T}\right\|_{\mathcal{B}(\mathbf{R}(A^{1/2}))}^p.$$
Thus, an application of Lemma \ref{lem2}(i) gives $\|T\|_A^p=\|T^p\|_A$. So, by taking into consideration \eqref{cc2} we get
\begin{equation*}
\lim_{n\rightarrow\infty} \Big\|(T^p-\lambda^p\|T^p\|_AI)x_n\Big\|_A=0,
\end{equation*}
for every $p\in\mathbb{N}$. Therefore, by Theorem \ref{Dau}, we get $T^p\parallel_A I$.

Now, the implications (2)$\Rightarrow$(3) and (3)$\Rightarrow$(1) follow immediately by using the equivalences  of Theorem \ref{th.19}.
\end{proof}

\begin{remark}
Notice that the equivalence $(1)\Leftrightarrow (2)$ in Corollary \ref{cor22} holds also for $A$-bounded operators.
\end{remark}

A special case of $A$-seminorm-parallelism between an $A$-bounded operator $T\in \mathcal{B}_{A^{1/2}}(\mathcal{H})$ and the identity operator, is the following equation:
\begin{equation}\label{daugavet}
\|T+ I\|_A=\|T\|_A+1.
\end{equation}
If $T\in \mathcal{B}_{A^{1/2}}(\mathcal{H})$ and satisfies \eqref{daugavet}, we shall say that $T$ satisfies the $A$-Daugavet equation. We remind here that the first person who study the equation \eqref{daugavet} for $A=I$ was I. K. Daugavet \cite{IKDaugavet}, which is one useful property in solving a variety of problems in approximation theory. Abramovich et al. \cite{AAB} proved that $T\in\mathcal{B}(\mathcal{H})$ satisfies the $I$-Daugavet equation (respect to the uniform norm) if and only if $\|T\|$ lies in the approximate point spectrum of $T$.

In the following theorem we shall characterize $A$-bounded operators which satisfy the $A$-Daugavet equation.
\begin{theorem}
Let $T\in \mathcal{B}_{A^{1/2}}(\mathcal{H})$. Then the following conditions are equivalent:
\begin{itemize}
  \item [(1)] $T$ satisfies the
$A$-Daugavet equation, i.e. $\|T+I\|_A=\|T\|_A+1$.
  \item [(2)] $\|T\|_A\in \overline{W_A(T)}.$
\item [(3)] $I \perp_A^{BJ} \|T\|_A I- T.$
\item [(4)] $T \perp_A^{BJ}  T-\|T\|_A I.$
\end{itemize}
\end{theorem}
\begin{proof}
(2) $\Rightarrow$ (1) Assume that $\|T\|_A\in \overline{W_A(T)}$. Then, there exits a sequence of $A$-unit vectors $\{x_n\}$ in $\mathcal{H}$ such that
$\lim\limits_{n\to \infty}\langle Tx_n, x_n\rangle_A=\|T\|_A.$ Thus
\begin{equation}\label{mm1}
\displaystyle{\lim_{n\rightarrow +\infty}}\Re(\langle Tx_n, x_n\rangle)_{A} = {\|T\|}_{A}.
\end{equation}
 Further, since
\begin{align*}
{\|T\|}^2_{A} + 2|{\langle Tx_n,x_n\rangle}_{A}| + 1
& \leq {\|T\|}^2_{A} + 2{\|Tx_n\|}_{A} + 1\\
& \leq {\|T\|}^2_{A} + 2{\|T\|}_{A} + 1
= ({\|T\|}_{A} + 1)^2,
\end{align*}
for all $n\in \mathbb{N}$, we get
\begin{align}\label{I.3.T.3.9}
\displaystyle{\lim_{n\rightarrow +\infty}}{\|Tx_n\|}_{A} ={\|T\|}_{A}.
\end{align}
Hence, by using \eqref{mm1} together with \eqref{I.3.T.3.9} we see that
\begin{align*}
({\|T\|}_{A} + 1)^2 &= \displaystyle{\lim_{n\rightarrow +\infty}}{\|Tx_n\|}^2_{A}
+ 2\displaystyle{\lim_{n\rightarrow +\infty}}\Re(\langle Tx_n, x_n\rangle_{A})
+ 1\\
& = \displaystyle{\lim_{n\rightarrow +\infty}}{\|(T + I)x_n\|}^2_{A}\leq {\|T + I\|}^2_{A}\leq ({\|T\|}_{A} + 1)^2.
\end{align*}
So ${\|T + I\|}_{A} = {\|T\|}_{A} + 1$.

(1) $\Rightarrow$ (2) Suppose that $\|T+I\|_A=\|T\|_A+1$. Then, by \eqref{semii} there exists a sequence of $A$-unit vectors $\{x_n\}$ in $\mathcal{H}$ such that
\begin{equation}\label{triangleeq}
\lim\limits_{n\to \infty} \|Tx_n+x_n\|_A=\|T\|_A+1.
\end{equation}
Since
$$
\|Tx_n+x_n\|_A\leq \|Tx_n\|_A+1\leq \|T\|_A+1,
$$
then, by using \eqref{triangleeq}, we conclude that
\begin{equation}\label{norm2}
\lim\limits_{n\to \infty} \|Tx_n\|_A=\|T\|_A.
\end{equation}
 On the other hand, since
 \begin{equation*}
 \|Tx_n+x_n\|_A^2=\|Tx_n\|_A^2+1 +2\Re(\langle Tx_n , x_n\rangle_A),
 \end{equation*}
for all $n\in \mathbb{N}$, then it follows from \eqref{triangleeq} together with \eqref{norm2} that
\begin{equation}\label{normre}
\lim\limits_{n\to \infty} \Re(\langle Tx_n , x_n\rangle_A)=\|T\|_A,
\end{equation}
for all $n\in \mathbb{N}$. Further, for every $n\in \mathbb{N}$, we see that
\begin{align*}
\Re^2(\langle Tx_n, x_n\rangle_{A})\leq\Re^2(\langle Tx_n, x_n\rangle_{A}) + \Im^2(\langle Tx_n, x_n\rangle_{A})
= |{\langle Tx_n, x_n\rangle}_{A}|^2 \leq {\|T\|}_{A}^2,
\end{align*}
and so by \eqref{normre}, we infer that
$\displaystyle{\lim_{n\rightarrow +\infty}}\Im(\langle Tx_n, Sx_n\rangle_{A}) = 0$. This yields, through \eqref{normre}, that
\begin{align*}
\displaystyle{\lim_{n\rightarrow +\infty}}{\langle Tx_n, x_n\rangle}_{A} = {\|T\|}_{A}.
\end{align*}
 Thus, we conclude that $\|T\|_A\in \overline{W_A(T)}$.

(1) $\Leftrightarrow$ (3) Assume that $T$ satifies the $A$-Daugavet equation. Then, by the equivalence between $(1)$ and $(2)$, we have $\|T\|_A\in \overline{W_A(T)}$. So, there exists a sequence of $A$-unit vectors $\{x_n\}$ in $\mathcal{H}$ satisfying
\begin{equation}\label{lambda1}
\lim_{n\rightarrow\infty} \langle Tx_n, x_n\rangle_A=\|T\|_A.
\end{equation}
This implies that
\begin{equation*}
\lim_{n\rightarrow\infty} \|Ix_n\|_A=\|I\|_A=1 \quad \text{ and }\quad \lim_{n\rightarrow\infty} \langle (T-\|T\|_A I)x_n, x_n\rangle_A=0,
\end{equation*}
then by Theorem \ref{zamaniABJ}, we have $I \perp_A^{BJ} \|T\|_A I- T$.  The converse is analogous.\\
(1) $\Leftrightarrow$ (4) Assume that $T$ satifies the $A$-Daugavet equation. Let $\{x_n\}$ a sequence of $A$-unit vectors  in $\mathcal{H}$ satisfying \eqref{lambda1}. Then
$$
\|T\|_A\geq \|Tx_n\|_A\geq | \langle Tx_n, x_n\rangle_A|\geq \|T\|_A-\epsilon,
$$
for any $\epsilon >0$ and $n$  large enough. Hence, $\lim\limits_{n\to \infty} \|Tx_n\|_A=\|T\|_A.$ Futhermore,
\begin{equation*}
\lim_{n\rightarrow\infty}  \langle Tx_n, (T-\|T\|_A I)x_n\rangle_A=\lim_{n\rightarrow\infty}   \|Tx_n\|_A^2 - \|T\|_A \langle Tx_n, x_n\rangle_A=0.
\end{equation*}
So, by Theorem \ref{zamaniABJ}, we deduce that $T\perp_A^{BJ}T-\|T\|_A I.$ Conversely, assume that $T\perp_A^{BJ}T-\|T\|_A I$. If $\|T\|_A=0$, then by using \eqref{newsemi} we see that the assertion $(1)$ holds trivially. Now, suppose that $\|T\|_A\neq0$. By Theorem \ref{zamaniABJ}, there exists a sequence of $A$-unit vectors $\{y_n\}$ in $\mathcal{H}$ such that
\begin{equation*}
\lim_{n\rightarrow\infty} \|Ty_n\|_A=\|T\|_A \quad \text{ and }\quad\lim_{n\rightarrow\infty}  \langle Ty_n, (T-\|T\|_A I)y_n\rangle_A=0.
\end{equation*}
So, it follows that
$$\lim\limits_{n\to \infty} \langle Ty_n, y_n\rangle_A=\frac{1}{\|T\|_A}\lim\limits_{n\to \infty} \|Ty_n\|_A^2=\|T\|_A, $$
i.e. $\|T\|_A\in \overline{W_A(T)}.$ Hence, by the equivalence (1)$\Leftrightarrow$(2), the assertion $(1)$ holds. Therefore, the proof is complete.
\end{proof}


\section{$A$-Bikhorff-James orthogonality and distance formulas}\label{s4}

First, we study some inequalities related to $\omega_A(\cdot)$ in order to obtain different bounds that will be useful in the study of $A$-Birkhoff-James orthogonality of operators and distance formulas.

It is useful to recall that the third author proved in \cite[Theorem 2.7.]{feki004} that for every $T\in \mathcal{B}_{A^{1/2}}(\mathcal{H})$ we have
\begin{equation}\label{power2}
\omega_A^2(T)\leq \frac12(\omega_A(T^2)+\|T\|_A^2).
\end{equation}
\begin{remark}
	\begin{enumerate}
		\item Let $\phi:[0, \infty)\to \mathbb{R}$ be any nondecreasing convex function or midpoint convex function. Clearly convexity implies midpoint-convexity. However, there exist midpoint-convex functions that are not
		convex. Such functions can be very strange and interesting. Then
		\begin{equation*}		\phi(\omega_A^2(T))\leq \phi\left(\frac12 [\omega_A(T^2)+\|T\|_A^2]\right)\leq \frac12\phi( w_A(T^2))+\frac12\phi(\|T\|_A^2).
		\end{equation*}
		Now,  we generalize inequality \eqref{power2} for any $r\geq 1$. Let $\phi(x)=x^r$ with $r\geq 1$ then
		\begin{equation*}
		\omega_A^{2r}(T)\leq \frac12(\omega_A^r(T^2)+\|T\|_A^{2r}).
		\end{equation*}
	\end{enumerate}
\end{remark}

For $T, S\in\mathcal{B}_{A^{1/2}}(\mathcal{H})$, the $A$-distance between $T$ and $S$ is defined by Zamani in \cite{zamani2019_2} as
$$d_A(T, \mathbb{C}S) := \displaystyle{\inf_{\gamma \in \mathbb{C}}}{\|T + \gamma S\|}_{A}.$$
In the following result, we prove an upper bound for the nonnegative quantity $\|T\|_A^2-\omega_A^2(T),$
with $T\in \mathcal{B}_{A^{1/2}}(\mathcal{H})$ related to $d_A(T, \mathbb{C}I)$.
\begin{theorem}
Let $T\in \mathcal{B}_{A^{1/2}}(\mathcal{H})$. Then,
	\begin{equation}\label{distance}
	\|T\|_A^2-\omega_A^2(T) \leq d_A^2(T, \mathbb{C}I).
	\end{equation}
\end{theorem}
\begin{proof}
Notice first that for any $a, b\in \mathcal{H}$ with $b\neq 0$, we have
	\begin{equation*}
	\inf_{\lambda \in \mathbb{C}} \|a-\lambda b\|^2=\frac{\|a\|^2\|b\|^2-|\langle a, b\rangle|^2}{\|b\|^2}.
	\end{equation*}
This implies that
	\begin{equation}\label{246}
	\|a\|^2\|b\|^2-|\langle a, b\rangle|^2 \leq \|b\|^2 \|a-\lambda b\|^2,
	\end{equation}
	for any $a, b\in \mathcal{H}$ and $\lambda \in \mathbb{C}$. Let $x,y\in \mathcal{H}$ and $\lambda \in \mathbb{C}$. By choosing $a=A^{1/2}x$ and $a=A^{1/2}y$ in \eqref{246} we obtain
	\begin{equation}\label{cota1}
	\|x\|_A^2\|y\|_A^2-|\langle x, y\rangle_A|^2 \leq \|y\|_A^2 \|x-\lambda y\|_A^2,
	\end{equation}
	Now, we choose in \eqref{cota1} $x=Tz$ and $y=z$ with $z\in \mathcal{H}, \|z\|_A=1$ we get
	\begin{equation*}
	\|Tz\|_A^2-|\langle Tz, z\rangle_A|^2 \leq  \|Tz-\lambda z\|_A^2,
	\end{equation*}
	By taking the supremum over $z\in \mathcal{H}$, $\|z\|_A=1$ implies that
	\begin{equation*}
	\|T\|_A^2-\omega_A^2(T) \leq  \inf_{\lambda \in \mathbb{C}} \|T-\lambda I\|_A^2.
	\end{equation*}
This finishes the proof of the theorem.
\end{proof}
\begin{remark}
By combining \eqref{power2} together with \eqref{distance} we obtain
\begin{equation*}
\omega_A^2(T)-\omega_A(T^2)\leq \frac{1}{2}\left(\|T\|_A^2-\omega_A(T^2)\right)\leq \|T\|_A^2-\omega_A(T^2)\leq d_A^2(T, \mathbb{C}I),
\end{equation*}
for any $T\in \mathcal{B}_{A^{1/2}}(\mathcal{H})$.
\end{remark}

%

We recall from \cite{zamani2019_2} that the $A$-minimum modulus of an operator $T\in\mathcal{B}_{A^{1/2}}(\mathcal{H})$ is given by
\begin{align*}
m_A(T) =\inf\Big\{{\|Tx\|}_A\,; \,\;x\in \mathcal{H}, {\|x\|}_A = 1\Big\}.
\end{align*}
This concept is useful in characterize the $A$-Bikhorff-James orthogonality in $\mathcal{B}_{A^{1/2}}(\mathcal{H}).$
\begin{thqt}  (\cite[Theorem 2.2]{zamani2019_2})\label{zamaniABJ2}
Let $T,S\in \mathcal{B}_{A^{1/2}}(\mathcal{H})$.Then $T\perp_A^{BJ} S$ if and only if
$$\|T+\gamma S\|_A^2\geq \|T\|_A^2+|\gamma|^2m_A^2(S)\; \text{ for all }\;\gamma \in \mathbb{C}.$$
\end{thqt}

Let $T,S\in \mathcal{B}_{A^{1/2}}(\mathcal{H})$ with $m_A(S)>0$. Then, by Theorem \ref{zamaniABJ2} there exist a unique $t_0\in\mathbb{C},$ such that
\begin{equation}\label{centermass}
\|(T-t_0S)+\gamma S\|_A^2\geq \|(T-t_0S)\|_A^2+|\gamma|^2m_A^2(S)
\end{equation}
In \cite{stampfli}, for $T\in \mathcal{B}(\mathcal{H})$, Stampfli defined the center of mass of $T$ to be the scalar $c(T)$ that satisfies the
equality
\begin{equation*}
\|T-c(T)I\|=d_I(T,\mathbb{C}I).
\end{equation*}
Given $T,S\in \mathcal{B}_{A^{1/2}}(\mathcal{H})$ with $m_A(S)>0$, we define the $A$-center of mass of  $T$ relatively to $S$ to be the unique point $t_0$, and designate it by $c_A(T,S)$. That is
$$
\|T-c_A(T,S) S\|_A=d_A(T, \mathbb{C}S).
$$
 In \cite[Theorem 3.4]{zamani2019_2}, Zamani proved that if $T, S\in \mathcal{B}_{A^{1/2}}(\mathcal{H})$ with $m_A(S)>0$, then
\begin{equation}\label{distanceformula}
d_A^2(T, \mathbb{C}S)=\sup_{\|x\|_A=1} \left(\|Tx\|_A^2-\frac{|\langle Tx, Sx\rangle_A|^2}{\|Sx\|_A^2}\right).
\end{equation}
One of the methods to compute the center of mass of an operator is Williams's theorem \cite{Williams}. However, it is not usually easy to determine the exact value of it even in the finite dimensional case. In what follows we investigate how to determine explicitly the number $c_A(T,S)$.
\begin{theorem}
Let $T,S\in \mathcal{B}_{A^{1/2}}(\mathcal{H})$ with $m_A(S)>0$ then
\begin{equation*}
c_A(T,S)=\lim\limits_{n\to +\infty}\frac{\langle Tx_n, Sx_n\rangle_A}{\|Sx_n\|_A^2},
\end{equation*}
where  $\{x_n\}$ be a sequence of $A$-unit vectors, approximating the supremum in \eqref{distanceformula}.
\end{theorem}
\begin{proof}By the hypothesis, $m_A(S)>0$, we can conclude that $\|Sx\|_A\geq m_A(S)>0$ for all $x\in \mathcal{H}$ with $\|x\|_A=1.$
For sake of simplicity we denote $c_A=c_A(T, S)$. Let $\{x_n\}$ be a sequence of $A$-unit vectors, approximating the supremum in \eqref{distanceformula}.  Then
\begin{align*}
&\left|\frac{\langle Tx_n, Sx_n\rangle_A}{\|Sx_n\|_A}-c_A\|Sx_n\|_A\right|^2\nonumber\\
&=\frac{|\langle Tx_n, Sx_n\rangle_A|^2}{\|Sx_n\|_A^2}-2\Re\langle Tx_n, c_ASx_n \rangle_A+|c_A|^2\|Sx_n\|_A^2\nonumber\\
&=\|(T-c_AS)x_n\|_A^2-\|Tx_n\|_A^2+\frac{|\langle Tx_n, Sx_n\rangle_A|^2}{\|Sx_n\|_A^2}\nonumber\\
&\leq \|(T-c_A S)\|_A^2 -\|Tx_n\|_A^2+\frac{|\langle Tx_n, Sx_n\rangle_A|^2}{\|Sx_n\|_A^2}.
\end{align*}
As the operator $S$ is $A$-bounded from below, we obtain the following inequality
\begin{equation*}
\left|\frac{\langle Tx_n, Sx_n\rangle_A}{\|Sx_n\|_A^2}-c_A\right|\leq \frac{1}{m_A(S)}\left|\frac{\langle Tx_n, Sx_n\rangle_A}{\|Sx_n\|_A}-c_A\|Sx_n\|_A\right|\to 0.
\end{equation*}
\end{proof}
 Further, if $S=I$, then
\begin{equation*}
c_A(T,I)=\lim\limits_{n\to +\infty}\langle Tx_n, x_n\rangle_A,
\end{equation*}
where  $\{x_n\}$ be a sequence of $A$-unit vectors, approximating the supremum in \eqref{distanceformula}.
\begin{corollary}
Let $T\in \mathcal{B}_{A}(\mathcal{H})$ with $m_A(T^{\sharp_A})>0$ then
\begin{equation*}
c_A(T,T^{\sharp_A})=\lim\limits_{n\to +\infty}\frac{\langle T^2x_n, x_n\rangle_A}{\|T^{\sharp_A}x_n\|_A^2},
\end{equation*}
where  $\{x_n\}$ be a sequence of $A$-unit vectors, approximating the supremum in \eqref{distanceformula}.
\end{corollary}

In particular is $T$ is $A$-normal, i.e. $T^{\sharp_A}T=TT^{\sharp_A}$ with $m_A(T)>0$,  as $\left|\langle Tx_n, T^{\sharp_A}x_n\rangle_A\right|\leq \|Tx_n\|_A\|T^{\sharp_A}x_n\|_A= \|T^{\sharp_A}x_n\|_A^2,$ we may deduce the inequality $|c_A(T,T^{\sharp_A})|\leq 1$.\\

Using \eqref{centermass} and mimicking the proof in \cite{BaBou}, we obtain the following continuity theorem.
\begin{corollary}
Let $T,S \in \mathcal{B}_{A^{1/2}}(\mathcal{H})$ with $m_A(S)>0$. Then the application
$$
T\to c_A(T, S)
$$ is uniformly continuous.
\end{corollary}

In 1981 M. Fujii and S. Prasanna proved  that for any $T\in \mathcal{B}(\mathcal{H})$ the closed circular disc centered at Stampfli's center of mass  and with radius $M_T=d_I(T, \mathbb{C}I)$ contains  the  numerical range  of $T$.  Now, we extend this statement  for the class of
$A$-bounded operators as follows.
\begin{theorem}
Let $T\in \mathcal{B}_{A^{1/2}}(\mathcal{H})$. Then
\begin{equation*}
 W_A(T)\subseteq D\Big(c_A(T, I), d_A(T, \mathbb{C}I)\Big),
\end{equation*}
where $D(\lambda_0, r_0)=\{\lambda\in \mathbb{C}\,;\; |\lambda - \lambda_0|\leq r_0\}$ for any $\lambda_0\in\mathbb{C}$ and $r_0>0$.
\end{theorem}
\begin{proof}
We split the proof in two cases.

Case 1: $c_A(T, I)=0$ i.e. $d_A(T, \mathbb{C}I)=\|T\|_A.$
Then for any $x\in \mathcal{H}$ with $\|x\|_A=1$, we have
\begin{equation}\label{centerzero}
|\langle Tx, x\rangle_A|\leq \omega_A(T)\leq \|T\|_A=d_A(T, \mathbb{C}I).
\end{equation}
Case 2: $c_A(T, I)\neq0$ i.e. $d_A(T, \mathbb{C}I)=\|T-c_A(T,I)I\|_A.$ Let $T_0:=T-c_A(T,I)I$ then $T_0 \in \mathcal{B}_{A^{1/2}}(\mathcal{H})$ and $c_A(T_0, I)=0.$ Applying \eqref{centerzero}, we obtain for any $x\in \mathcal{H}, \|x\|_A=1$
\begin{equation*}
|\langle Tx, x\rangle_A-c_A(T,I)|=|\langle T_0x, x\rangle_A|\leq \|T_0\|_A=d_A(T, \mathbb{C}I).
\end{equation*}
This completes the proof.
\end{proof}

\begin{proposition}
Let $T \in \mathcal{B}_{A^{1/2}}(\mathcal{H})$ then
\begin{equation}\label{compdist}
	d_A(T,\mathbb{C}I)\leq \|T\|_A d_A(I,\mathbb{C}T).
\end{equation}
\end{proposition}
\begin{proof}
Let $x\in \mathcal{H}$ with $\|x\|_A=1$. Then
\begin{equation*}
\alpha_A(T)\|Tx\|_A\leq   |\langle Tx,x\rangle_A|,
\end{equation*}
where $\alpha_A(T)=\inf\left\{\frac{|\langle Ty,y\rangle_A|}{\|Ty\|_A}: \|Ty\|_A\neq0, \|y\|_A=1\right\}$ if $\|T\|_A\neq0$ or $\alpha_A(T)=0$  if $\|T\|_A=0.$ Then
\begin{equation*}
\|Tx\|_A^2-|\langle Tx,x\rangle_A|^2\leq \left(1-\alpha_A^2(T)\right)\|Tx\|_A^2\leq d_A^2(I, \mathbb{C}T)\|Tx\|_A^2.
\end{equation*}
Now calculating the supremum of the both sides, over all $x\in \mathcal{H}$ with $\|x\|_A=1$, we complete the proof.
\end{proof}

From  \eqref{distance} and \eqref{compdist}, we obtain
\begin{equation*}
 \|T\|_A^2-\omega_A^2(T)\leq d_A^2(T, \mathbb{C}I)\leq \|T\|_A^2 d_A^2(I,\mathbb{C}T).
\end{equation*}
\begin{corollary}\label{symmetry of BJ}
Let $T \in \mathcal{B}_{A^{1/2}}(\mathcal{H})$.  If   $T\perp_A^{BJ} I$, then $I\perp_A^{BJ} T$.
\end{corollary}
\begin{proof}
By \eqref{compdist}, we have
\begin{equation*}
\|T\|_A=d_A(T,\mathbb{C}I)\leq \|T\|_A d_A(I,\mathbb{C}T).
\end{equation*}
So, if $\|T\|_A\neq 0$, then $1\leq  d_A(I,\mathbb{C}T) \leq \|I\|_A=1$, i.e. $d_A(I,\mathbb{C}T)=\|I\|_A=1$.

On the other hand, if $\|T\|_A=0$ then $\|Tx\|_A=0$ for all $x\in \mathcal{H}, \|x\|_A=1$. From \cite[Theorem 3.4]{zamani2019_2}, we have that
$$
d_A^2(I, \mathbb{C}T)=\sup\{\|Ix\|_A^2: \|x\|_A=1\}=1=\|I\|_A.
$$
In conclusion, in both cases,  we obtain that  $I\perp_A^{BJ} T$.
\end{proof}
The converse of the previous result is false in general, as we see in the next example

\begin{example}
Consider in $\mathcal{H}=\mathbb{C}^3$ with  the usual uniform norm and  let $\{e_1, e_2, e_3\}$ be the canonical basis for $\mathcal{H}$.\\
Let $A=\begin{pmatrix}
1&0&0\\
0&1&0\\
0&0&0
\end{pmatrix}.$
Then $A=P_{\mathcal{M}}$ the orthogonal projection on $\mathcal{M}=gen\{e_1, e_2\}$ and  $A^2=A^*=A.$\\
Consider $T=\begin{pmatrix}
2&0&0\\
0&-1&0\\
0&0&1
\end{pmatrix}\in \mathcal{B}_{A^{1/2}}(\mathcal{H})$.\\
Let $x=\alpha e_1+\beta e_2 +\gamma e_3 \in \mathcal{H}$ then
\begin{equation*}
\|x\|_A^2=\|(\alpha, \beta, \gamma)\|_A^2=\langle x, x\rangle_A=\langle Ax, Ax\rangle=\|Ax\|^2=|\alpha|^ 2+|\beta|^2=\|(\alpha, \beta)\|^2.
\end{equation*}
Observe that $\|(\alpha, \beta, \gamma)\|_A^2=1$ if and only if $\|(\alpha, \beta)\|^2=1.$\\
Now
\begin{eqnarray*}
\|T\|_A^2&=&\sup \{\|Tx\|_A^2: x\in \mathbb{C}^3, \|x\|_A=1\}=\sup \{\|ATx\|^2: x\in \mathbb{C}^3, \|x\|_A=1\}\nonumber\\
&=&\sup \{\|\overline{T}x\|^2: \overline{x}\in \mathbb{C}^2, \|\overline{x}\|=1\}=\|\overline{T}\|^2=4,
\end{eqnarray*}
where  $\overline{T}=\begin{pmatrix}
2&0\\
0&-1\\
\end{pmatrix}\in \mathcal{B}(\mathbb{C}^2).$\\
If $I_n$ denotes the identity operator in $\mathcal{B}(\mathbb{C}^n)$, then
$$
\inf_{\lambda \in \mathbb{C}}\|T-\lambda I_3\|_A=\inf_{\lambda \in \mathbb{C}}\|\overline{T}-\lambda I_2\|=\frac{3}{2}<\|T\|_A=2,
$$
i.e. $T$ is not $A$-Birkhoff-James to $I_3.$ On the other hand,
$$
\inf_{\lambda \in \mathbb{C}}\|I_3-\lambda T\|_A=\inf_{\lambda \in \mathbb{C}}\|I_2-\lambda\overline{T}\|=1=\|I_3\|_A=1,
$$
that is $I_3 \perp_A^{BJ} T.$
\end{example}

The following result relates $A$-Birkhoff-James orthogonality with the attainment of the lower bound of the $A$-Davis-Wielandt radius.
\begin{theorem}\label{dw attains lower}
	Let $T\in\mathcal{B}_{A^{1/2}}(\mathcal{H})$ such that
	$d\omega_A(T)=\max\{\omega_A(T),\|T\|_A^2\}.$
	Then $T\perp_A^{BJ} I$.
\end{theorem}	
\begin{proof}
	We separate in two different cases.
	
	Case 1: Suppose $d\omega_A(T)=\|T\|_A^2$ and take a sequence of unitary vectors $\{y_n\}_{n\in \mathbb{N}}$ such that
	$\lim\limits_{n\to +\infty}\|Ty_n\|_A^2=\|T\|_A^2$. Then
	$$\|Ty_n\|_A^2\leq \sqrt{\left| \left\langle Ty_n,y_n \right\rangle_A \right|^2 +\|Ty_n\|_A^4}\leq d\omega_A(T)= \|T\|_A^2,$$
	therefore, $\lim\limits_{n\to +\infty}\left| \left\langle Ty_n,y_n \right\rangle_A\right|^2=0$
	and $0\in W_A(T,I)$. By Th. \ref{zamaniABJ} this is equivalent to
	$T\perp_{BJ}^AI.$\\
	Case 2: Suppose $d\omega_A(T)=\omega_A(T)$ and take a sequence of unitary vectors $\{z_n\}_{n\in \mathbb{N}}$ such that
	$\lim\limits_{n\to +\infty}\left| \left\langle Tz_n,z_n \right\rangle_A \right|=\omega_A(T)$. Then
	$$\left| \left\langle Tz_n,z_n \right\rangle_A \right|\leq \sqrt{\left| \left\langle Tz_n,z_n \right\rangle_A \right|^2 +\|Tz_n\|_A^4}\leq d\omega_A(T)= \omega_A(T),$$
	therefore, $\lim\limits_{n\to +\infty} \|Tz_n\|_A^4=0$. But
	$$\left| \left\langle Tz_n,z_n \right\rangle_A \right|\leq \|Tz_n\|_A\to 0,$$
	thus $\omega_A(T)=0$ and $\|T\|_A=0\leq \|T+\lambda I\|_A$ for every $\lambda\in \mathbb{C}$.
\end{proof}
We arrive to the next conclusion as a combination of Corollary \ref{symmetry of BJ} and Theorem \ref{dw attains lower}.
\begin{corollary}
	Let $T\in\mathcal{B}_{A^{1/2}}(\mathcal{H})$ such that
	$d\omega_A(T)=\max\{\omega_A(T),\|T\|_A^2\}.$
	Then $T\perp_A^{BJ} I$ and $I\perp_A^{BJ} T$.
\end{corollary}

\begin{remark}
	If $T=x\otimes_Ay$ with $\|x\|_A,\|y\|_A\neq 0$ the attainment of the lower bound of $d\omega_A(T)$ implies that $x\perp_Ay\ \text{or}\ A^{1/2}x\perp A^{1/2}y.$
	By Lemma \ref{lemonerank}, this is equivalent to
	$$\omega_A(x\otimes_Ay)=\frac 1 2 \|x\|_A\|y\|_A$$
	(i.e. the attaiment of the lower bound of $\omega_A(T)$). Indeed, first observe that $$\left|\langle  x, \frac{y}{\|y\|_A}\rangle _A \langle  \frac{y}{\|y\|_A}, y\rangle _A\right|=\left|\frac{1}{\|y\|_A^2}\langle x,y\rangle_A \|y\|_A^2\right|=|\langle x,y\rangle_A|$$
	On the other hand
	$$\left\| (x\otimes_Ay)\frac{y}{\|y\|_A}\right\|_A^4=\frac{1}{\|y\|_A^4}\left\|\left\langle y,y \right\rangle_Ax  \right\|_A^4=\|y\|_A^4\|x\|_A^4.$$
	Then,
	$$\sqrt{\left| \langle (x\otimes_Ay)\frac{y}{\|y\|_A}, \frac{y}{\|y\|_A}\rangle_A\right| ^2+\left\|(x\otimes_Ay)\frac{y}{\|y\|_A}\right\| _A^4}=\sqrt{|\left\langle x,y\right\rangle_A|^2+\|y\|_A^4\|x\|_A^4}$$
	and
	\begin{equation*}
	d\omega_A^2(x\otimes_Ay)\geq |\left\langle x,y\right\rangle_A|^2+\|y\|_A^4\|x\|_A^4.
	\end{equation*}
	If $(\|x\|_A\|y\|_A)^2=\|x\otimes_A y\|_A^2=d\omega_A(x\otimes_A y)$,
	$$\|x\|_A^4\|y\|_A^4=d\omega_A^2(x\otimes_Ay)\geq |\left\langle x,y\right\rangle_A|^2+\|x\|_A^4\|y\|_A^4, $$
	therefore $\left\langle x,y\right\rangle_A=0$.
\end{remark}


\end{document}